\documentclass[11pt]{amsart}
\usepackage[margin=1in]{geometry}
\usepackage{comment}
\usepackage{amsmath,amssymb,mathtools,booktabs,array}
\usepackage[T1]{fontenc}
\usepackage{lmodern}
\usepackage[colorlinks=true,linkcolor=blue,citecolor=blue,urlcolor=blue]{hyperref}
\hypersetup{pdftitle={Multiplicity one for unitary similitude Ginzburg--Rallis models over local fields}}
\newtheorem{theorem}{Theorem}[section]
\newtheorem{proposition}[theorem]{Proposition}
\newtheorem{lemma}[theorem]{Lemma}
\theoremstyle{definition}
\newtheorem{definition}[theorem]{Definition}
\theoremstyle{remark}
\newtheorem{remark}[theorem]{Remark}
\newcommand{\R}{\mathbb R}
\newcommand{\C}{\mathbb C}
\newcommand{\HH}{\mathcal H}
\newcommand{\g}{\mathfrak g}
\newcommand{\s}{\mathfrak s}
\newcommand{\uu}{\mathfrak u}
\newcommand{\mm}{\mathfrak m}
\newcommand{\kk}{\mathfrak k}
\newcommand{\cgen}{\mathcal C^{-\infty}}
\newcommand{\ctemp}{\mathcal C^{-\xi}}
\newcommand{\Tr}{\operatorname{Tr}_{E/F}}
\DeclareMathOperator{\GU}{GU}
\DeclareMathOperator{\GL}{GL}
\DeclareMathOperator{\Hom}{Hom}
\DeclareMathOperator{\tr}{tr}
\DeclareMathOperator{\Ree}{Re}
\DeclareMathOperator{\Imm}{Im}
\DeclareMathOperator{\diag}{diag}
\DeclareMathOperator{\rank}{rank}
\DeclareMathOperator{\Ad}{Ad}
\DeclareMathOperator{\Lie}{Lie}
\DeclareMathOperator{\supp}{supp}
\DeclareMathOperator{\SO}{SO}
\DeclareMathOperator{\OO}{O}
\title[Unitary Ginzburg--Rallis multiplicity one]{Multiplicity one for unitary similitude\ Ginzburg--Rallis models over local fields}
\author{Yuxuan Ge, Sagnik Mukherjee, and Junyi Zhai}
\date{}
\numberwithin{equation}{section}
\begin{document}
\begin{abstract}
We prove multiplicity at most one for the unitary similitude
Ginzburg--Rallis model over local fields of characteristic zero,
for both the quasi-split and non-quasi-split groups. The theorem
holds for all irreducible smooth representations over
non-Archimedean fields and all irreducible Casselman--Wallach
representations over the real field. We first describe the
decomposition into parabolic double cosets and compute the
required stabilizers.
Over a non-Archimedean field we use character obstructions,
localization, and transpose symmetry. Over the real field we use
a Casimir identity and explicit normal vectors to prove the
corresponding statement for distributions with prescribed
infinitesimal character. The Gelfand--Kazhdan criterion and a
dualizing involution then give the multiplicity bound.
\end{abstract}
\maketitle
\tableofcontents

\clearpage
\section{Introduction}\label{sec:introduction}

Multiplicity one problems concern the uniqueness of equivariant
linear functionals on irreducible admissible representations.
For a reductive group $G_0$ over a local field $F$ of
characteristic zero, a closed subgroup $J\subset G_0(F)$, and a
smooth character $\theta$ of $J$, the question is whether
\[
 \dim\Hom_J(\pi,\theta)\leq 1
\]
for every irreducible admissible representation $\pi$ of $G_0(F)$.
Here $\Hom_J(\pi,\theta)$ consists of the linear functionals
$\ell$ satisfying $\ell(\pi(j)v)=\theta(j)\ell(v)$ for all
$j\in J$ and $v\in V_\pi$.\footnote{In the Archimedean case
we work in the Casselman--Wallach category and require the
functionals to be continuous.}
Such uniqueness is a basic input in factoring automorphic periods
into local functionals and relating them to special values of
automorphic $L$-functions. This connection motivates the study of
multiplicity one for a range of local models. One such example
is the
Ginzburg--Rallis model for $\GL_6$, introduced in~\cite[Section~4]{GR} to
relate period integrals to the nonvanishing of central values of
the exterior cube $L$-function. Its defining subgroup combines
the unipotent radical of the parabolic of type $(2,2,2)$ with
the diagonal copy of $\GL_2$ stabilizing a generic character
of that radical. The corresponding unitary similitude model has
a similar structure. We give the precise
definitions of the subgroup and its character $\theta$ in
Section~\ref{subsec:main-results}. Its expected relation to
exterior cube central values is discussed
in~\cite[Remark~1.5]{WZunitary}. In this paper we prove
multiplicity one for this unitary similitude model attached to
a quadratic field extension of local fields of characteristic
zero, for both the quasi-split and the non-quasi-split group.

To place this theorem in context, we first recall the conjectures
relating Ginzburg--Rallis periods to central $L$-values and the
local multiplicity results they motivate, beginning with the
general linear model and then passing to its unitary analogues.

\subsection{Related work}

The Ginzburg--Rallis model was introduced
in~\cite[Section~4]{GR} to relate the nonvanishing of central
exterior cube $L$-values to automorphic periods. The resulting
conjecture, restated in~\cite[Conjecture~1.1]{JSZ}, predicts
that, under the prescribed central-character condition, a
partial twisted exterior cube $L$-value at $s=\tfrac12$ is
nonzero precisely when the corresponding period is nonzero
on a unique quaternionic inner form, including the split form.
A first local step toward this global problem is the uniqueness
of the corresponding equivariant functional, formulated
in~\cite[Conjecture~1.2]{JSZ}.

The local uniqueness statement is now known over
all local fields of characteristic zero. The non-Archimedean
quaternionic case is proved in~\cite{NienThesis}, the
Archimedean cases for both forms in~\cite[Theorem~1.3]{JSZ},
and the split non-Archimedean case in~\cite[Theorem~1.1]{JLX}.
The distributional arguments in the latter two papers provide
the main methodological precedents for our proofs.

Multiplicity formulas then determine when the unique functional
exists. The multiplicity formula proved for supercuspidal
representations over non-Archimedean fields
in~\cite[Theorem~1.3]{WanGeometric} is extended
in~\cite[Theorem~1.2]{WanTempered} to every irreducible tempered
representation over non-Archimedean fields and $\R$ with the
prescribed compatible central character:
\[
 m(\pi)+m(\pi_D)=1.
\]
Here $\pi_D$ is the Jacquet--Langlands transfer to $\GL_3(D)$,
where $D$ is the quaternion division algebra, and its multiplicity
is interpreted as zero when no transfer exists. This is the
tempered case of the local conjecture
in~\cite[Conjecture~4.6]{Jiang2008}, formulated for generic
representations in~\cite[Conjecture~1.1]{WanTempered}.
The proof uses a local relative trace formula in the
non-Archimedean case and parabolic induction and trilinear models
over $\R$~\cite[Section~6.2]{WanTempered}. Beyond the tempered
case, the identity holds for generic representations when
$\pi_D\ne0$~\cite[Theorem~1.2]{WanGeneric}; the full generic
statement in~\cite[Theorem~1.7 and Appendix~B]{WanGeneric}
assumes holomorphic continuation of generalized Jacquet
integrals. The generic results over $\C$
in~\cite[Theorem~1.4]{WanComplex} hold under the stated
conditions on the inducing parabolic and trilinear data.

The link with exterior cube $L$-values also leads to the unitary
similitude model considered here, through the integral
representation in~\cite[Propositions~2.2 and~3.3]{ZhangExterior}.
The direct local predecessors of our work are the tempered
multiplicity formulas for the unitary similitude and unitary
Ginzburg--Rallis models over non-Archimedean fields
in~\cite[Theorem~4.9 and Proposition~5.5]{WZunitary}.
The total multiplicity of a tempered Vogan packet is one for
the similitude model, assuming the packet classification
in~\cite[Conjecture~2.5 and Theorem~1.1]{WZunitary}, whereas
the unitary packet total is two~\cite[Theorem~1.2]{WZunitary}.
Uniqueness for all irreducible smooth representations of the
similitude group is stated as an unnumbered conjecture
in~\cite[Remark~1.4]{WZunitary}.

The unitary similitude model also belongs to the family of ten
strongly tempered spherical varieties in~\cite[Table~1]{WZ},
which includes the general linear Ginzburg--Rallis model.
In this broader setting, geometric multiplicity formulas for
all tempered representations are predicted
in~\cite[Conjecture~7.4(3)]{WanSpherical}.
Tempered multiplicity formulas are
proved for Models~1--9 over non-Archimedean fields and $\C$,
and for Models~1--4 over $\R$~\cite[Section~1.2 and Chapter~9]{WZ}.
A description of multiplicities using epsilon factors is proposed
in~\cite[Conjecture~2.10]{WZConjecture} for strongly tempered
spherical varieties without type $N$ roots.
For the split forms of Models~6--10, the recent
preprint~\cite[Theorem~1.1]{GeStronglyTempered} proves
multiplicity one for all irreducible smooth admissible
representations over non-Archimedean fields of characteristic
zero. The present paper treats Model~5, the unitary similitude
Ginzburg--Rallis model, for both Hermitian forms over
non-Archimedean fields and $\R$, without a temperedness
assumption.

\subsection{Main results}\label{subsec:main-results}

To state the theorem for both forms at once, we describe the
subgroup and character in terms of a Hermitian plane. The notations we use in this subsection, are defined with more details in Section \ref{sec:model}. Let $E/F$
be a quadratic
field extension of local fields of characteristic zero, and fix a
nontrivial unitary additive character $\psi$ of $F$. Let $H$ be a
nondegenerate Hermitian matrix of rank two, either hyperbolic or
anisotropic. With $w_3$ denoting the anti-diagonal permutation
matrix, put
\[
 W=w_3\otimes H,\qquad G=\GU(W)(F),\qquad K=\GU(H)(F).
\]
Here $\GU$ denotes the unitary similitude group defined in
Section~\ref{sec:model}.

The unipotent part of the model comes from the parabolic subgroup
$P=MU$ stabilizing a two-dimensional totally isotropic subspace, whose Levi factor is
$M\simeq\GL_2(E)\times K$. In the block coordinates of
Section~\ref{sec:model}, write $X$ for the first adjacent block of
$u(X,Y)\in U$ and set
\[
 \xi(u(X,Y))=\psi\bigl(\operatorname{Tr}_{E/F}(\tr X)\bigr).
\]
This character determines the reductive part as well: its
stabilizer in $M$ is
$D(K)=\{\diag(k,k,k):k\in K\}$. Thus the model subgroup is
$S=D(K)U$. To include the determinant and similitude twists, take
smooth multiplicative characters $\chi_E$ of $E^\times$ and
$\chi_F$ of $F^\times$ and define
\[
 \eta(D(k)u)=\chi_E(\det k)\chi_F(\lambda(k))\xi(u),
\]
where $\lambda$ is the similitude character. In the notation of
the opening paragraph, we have $G_0(F)=G$, $J=S$, and
$\theta=\eta$. With these data, our main result is the following.

\begin{theorem}\label{thm:main}
Suppose that $H$ is either hyperbolic or anisotropic. If $F$ is
non-Archimedean, let $\pi$ be an irreducible smooth complex
representation of $G$. If $F=\R$, let $\pi$ be an irreducible
Casselman--Wallach representation of $G$. Then
\[
                         \dim\Hom_S(\pi,\eta)\leq1.
\]
\end{theorem}
Notice that, in the real case the functionals are continuous. Also, no temperedness
assumption is imposed on $\pi$, and the multiplicative characters
$\chi_E$ and $\chi_F$ need not be unitary.

The two choices of $H$ give the quasi-split and non-quasi-split
Hermitian groups; $E$ remains a field in both cases. The
non-Archimedean assertion has no restriction on the residue
characteristic. In the Archimedean case a quadratic field extension
is necessarily $\C/\R$, and the theorem concerns the full groups
$\GU(3,3)$ and $\GU(4,2)$. In particular, $\GU(3,3)$ includes
the component with negative similitude multiplier.

We deduce Theorem~\ref{thm:main} from the distributional symmetry
stated in Theorem~\ref{thm:distribution} below, by applying the
Gelfand--Kazhdan criterion together with the contragredient
identification in Lemma~\ref{lem:dual}. To formulate this
symmetry, let $\tau$ be the anti-involution in
\eqref{eq:tau-common}: ordinary transpose in the quasi-split
coordinates and weighted transpose in the non-quasi-split
coordinates. The relevant equivariance under left and right
translation is encoded by
\[
 \HH=S\times S,\qquad (s_1,s_2)\cdot g=s_1g\tau(s_2),
 \qquad \chi=\eta\boxtimes\eta.
\]
In the real case, our symmetry statement also imposes an
eigenvalue condition for the Casimir operator $\Omega$ associated
with the trace form $B(X,Y)=\Ree\tr(XY)$ on $\Lie(G)$.
The distribution spaces and this operator are defined in
Section~\ref{sec:tools}.

\begin{theorem}\label{thm:distribution}
If $F$ is non-Archimedean, every $(\HH,\chi)$-equivariant
distribution on $G$ is $\tau$-invariant. If $F=\R$, every
$(\HH,\chi)$-equivariant tempered generalized function $f$ on
$G$ satisfying $\Omega f=cf$ for some $c\in\C$ is
$\tau$-invariant.
\end{theorem}

In the real case, the generalized matrix coefficients to which
the Gelfand--Kazhdan criterion~\cite[Proposition~10.1]{JSZ}
is applied are tempered generalized functions and joint
eigenvectors for the bi-invariant differential operators. They
therefore satisfy the Casimir eigenvalue condition in
Theorem~\ref{thm:distribution}. Here temperedness concerns these
functions and imposes no temperedness assumption on the
representations in Theorem~\ref{thm:main}. The next subsection
outlines the proof of Theorem~\ref{thm:distribution} and the
deduction of Theorem~\ref{thm:main}.

\subsection{Outline of the proof}

We prove the Theorem \ref{thm:distribution} by a vanishing argument. The
decomposition of $G$ into $P$--$\tau(P)$ double cosets allows us
to prove vanishing on successively larger open subsets, beginning
with the open cell. Once vanishing there is established,
the orbits on the boundary requires different arguments over non-Archimedean
fields and over $\R$; these are the two branches described in
Steps~4 and~5. After either branch, the final step deduces
multiplicity one from the resulting distributional symmetry.

\paragraph{\textbf{1. Reduction to odd equivariant distributions.}}
The failure of transpose symmetry is measured by the difference
$f^-=f-f\circ\tau$ for an equivariant distribution $f$. Extend the
$\HH$-action to $\widetilde\HH=\HH\rtimes\langle\tau\rangle$,
where $\tau$ acts on $\HH$ by interchanging its two factors,
and extend $\chi$ to $\widetilde\chi$ by setting
$\widetilde\chi(\tau)=-1$, as in Section~\ref{sec:model}.
Then $f^-$ is $(\widetilde\HH,\widetilde\chi)$-equivariant.
We call such distributions odd. It suffices to prove that every
odd distribution under consideration vanishes. Over $\R$, the
Casimir commutes with $\tau$, so passing to $f^-$ preserves the
equation $\Omega f=cf$. In the remaining steps, $f$ denotes
an odd distribution in the required category.

\paragraph{\textbf{2. Parabolic double coset decomposition.}}
We control the support of an odd distribution by decomposing $G$
into finitely many $P$--$\tau(P)$ double cosets. By
Lemma~\ref{lem:rank-common}, these double cosets are smooth and
are classified by the ranks of the lower-right block submatrices
of sizes $4\times4$, $4\times2$, $2\times4$, and $2\times2$.
Writing $G_{abcd}$ for the double coset with these four ranks
equal to $a,b,c,d$, respectively, we obtain
\[
 \begin{array}{c|c}
 \text{hyperbolic }H&4222,\ 3221,\ 3111,\ 2110,\ 2000\\
 \text{anisotropic }H&4222,\ 3111,\ 2000.
 \end{array}
\]
The classification uses the dimension of the intersection of
two isotropic planes and the rank of the Hermitian pairing
between them. Witt extension then shows that these invariants
classify the double cosets over $F$. The open double coset is $O=G_{4222}=UM\tau(U)$.
Each parabolic double coset can contain many $\HH$-orbits,
whose stabilizers enter the subsequent vanishing arguments.

The rank conditions also give an increasing sequence of
invariant open subsets, starting with $O$. Once a distribution
vanishes on one of these open subsets, its restriction to the
next is supported on a single smooth stratum. This provides the
induction used in the boundary arguments. The open-cell
calculation supplies its starting point.

\paragraph{\textbf{3. The open cell.}}
On $O$, the Gauss decomposition first removes the two equal
unipotent character factors. Descent on
$M\simeq\GL_2(E)\times K$, with coordinate $z=ak^{-1}$,
then reduces the problem to $K$-conjugation-invariant
distributions on $\GL_2(E)$ when $F$ is non-Archimedean, and to
$K$-conjugation-invariant tempered generalized functions on
$\GL_2(\C)$ when $F=\R$. In the real case, we start with a
tempered generalized function on $G$, that is, a continuous
linear functional on Schwartz densities. The Gauss isomorphism
and equivariant restriction to the Nash slice preserve
temperedness (Sections~\ref{subsubsec:nash-functions}
and~\ref{subsubsec:real-localization}). For this smaller problem,
the two Hermitian forms lead to different arguments.

For hyperbolic $H$, Hilbert 90 gives $K=E^\times\GL_2(F)$
after conjugation. In the coordinates $Z=A+\delta B$, with
$\delta\in E^\times$ satisfying $\bar\delta=-\delta$, it suffices
to prove that every distribution on $X=M_2(F)^2$ (a tempered
generalized function when $F=\R$) invariant under simultaneous
$\GL_2(F)$-conjugation of $A$ and $B$ is invariant under
$(A,B)\mapsto(A^t,B^t)$. For non-Archimedean $F$, Bernstein
localization (Theorem~\ref{thm:localization}), applied to $\det_E$,
reduces the assertion on $\GL_2(E)$ to that on $X$.
The scalar translates of $Y=\GL_2(F)^2$ form an invariant
open cover of $X$, reducing the assertion to $Y$, where it
follows from~\cite[Proposition~4.5]{Prasad}.
Over $\R$, conjugation on $V=\mathfrak{sl}_2(\R)$, equipped
with the trace form, has image $\SO(V)$. By
\cite[Proposition~1.5]{JSZ}, every $\SO(V)$-invariant tempered
generalized function on $V^2$ is $\mathrm O(V)$-invariant,
hence transpose-invariant. Since both actions fix the scalar
coordinates, this proves the assertion on $X$. Since $\det_{\C}$
is an invariant Nash map and $\C^\times$ is Nash open, the
localization statement~\eqref{eq:real-localization} then gives
the assertion on $\GL_2(\C)$; see~\cite[Lemma~2.8]{JSZ}. 

For anisotropic $H$, we must prove that every
$K$-conjugation-invariant distribution $T$ on $\GL_2(E)$ is
invariant under $\tau_2(z)=H^{-1}z^tH$. Let $\mathcal A\varphi$
denote the average of a test function $\varphi$ over the compact
group $K/E^\times$. Since $T(\varphi)=T(\mathcal A\varphi)$,
it suffices to show $\mathcal A(\varphi\circ\tau_2)=\mathcal A\varphi$.
This holds because averaging commutes with $\tau_2$, while
$\mathcal A\varphi$ is constant on the conjugacy orbits
preserved by $\tau_2$.

Thus odd distributions vanish on $O$ in both settings, and it
remains to eliminate their support on the boundary.

\paragraph{\textbf{4. The boundary for $F$-non-Archimedean.}}
Over a non-Archimedean field, the boundary argument uses
character obstructions and transpose-invariance. Frobenius
reciprocity implies that an orbit carries no nonzero equivariant
distribution if its stabilizer contains a unipotent element on
which the character is nontrivial; see~\cite[\S6.11]{BZ} and
Lemma~\ref{lem:padic-orbits}. On most remaining pieces, explicit
elements carry each point to its transpose with character one.
By~\eqref{eq:closed-support-tools}, a distribution vanishing
outside a closed subset is a distribution on that subset.
This lets us prove vanishing successively along a closed
invariant filtration. However, one slice requires a further
support calculation.

This slice occurs in $T=G_{3221}$ in the quasi-split case.
Lemma~\ref{lem:3221-filtration} gives a closed filtration
$T\supset C\supset C_0\supset Z\supset\varnothing$,
defined by the stabilizer characters.
By Lemma~\ref{lem:padic-orbits}, the nontriviality of $\chi$
on suitable unipotent stabilizers forces every odd equivariant
distribution on $T\setminus C$ or $C_0\setminus Z$ to vanish.
Explicit transpose symmetry gives the same vanishing on $Z$.
By Lemma~\ref{lem:3221-normal-form}, the remaining piece
$C\setminus C_0$ has a normal form with parameters
$(\zeta,x,y)\in E^\times\times F^2$. The induced action is
\[
 (x,y)\longmapsto(tx,t^{-1}y),\quad t\in F^\times,
 \qquad \tau:(x,y)\longmapsto(y,x).
\]
We prove directly that every distribution invariant under this
scaling is invariant under interchange of the coordinates. The
proof uses Bernstein localization (Theorem~\ref{thm:localization})
for the product $xy$, first treating its nonzero fibers and then
the two axes and the origin. Transpose exchanges distinct scaling orbits on
the punctured axes, so this step is not an orbitwise symmetry
argument. This completes the remaining support calculation
in the quasi-split boundary argument.

The non-quasi-split
case has no $3221$ or $2110$ stratum and requires only the
character and transpose calculations on $3111$ and $2000$.
Together with the open-cell result, this proves the
non-Archimedean distribution theorem.

\paragraph{\textbf{5. The boundary for $F=\R$.}}
The argument for $F=\R$ also starts from $f|_O=0$, but a
distribution supported on a stratum may involve derivatives in
directions normal to it. To rule out such support, we use the
condition $\Omega f=cf$ in Theorem~\ref{thm:distribution},
where $\Omega$ is the Casimir differential operator for the
trace form $B(X,Y)=\Ree\tr(XY)$, together with the support
lemmas in Section~\ref{subsec:real-tools}.

We first apply these tools to $T=G_{3221}$ in the quasi-split
case. The restriction of $f$ to the open set $G'=O\cup T$
is supported on $T$. On $T\setminus C$ and $C_0\setminus Z$,
we construct smooth families of unipotent stabilizing elements
on which $\chi$ is nontrivial. Their unipotent action on normal
directions lets Lemma~\ref{lem:real-unipotent} exclude equivariant
distributions supported there, including all normal derivatives.
On $C\setminus C_0$ and $Z$, explicit normal vectors of nonzero
squared length allow Lemma~\ref{lem:metric-support} to exclude
such support using $(\Omega-c)f=0$.
Lemma~\ref{lem:real-filtration} combines these conclusions along
the filtration~\eqref{eq:filtration}, giving $f|_{G'}=0$.

It remains to exclude support on the parabolic strata
$G_{3111}$, $G_{2110}$, and $G_{2000}$.
The Casimir identity~\eqref{eq:casimir} rewrites $\Omega f=cf$ as
\[
                  (2i\kappa\mathcal X+D_M^L-c)f=0,
\]
where $\psi(t)=e^{i\kappa t}$ with $\kappa\ne0$, $\mathcal X$
is a vector field transverse to these strata, and $D_M^L-c$ is a tangential
differential operator. The transverse term raises the highest order of
normal differentiation, whereas the tangential term does not.
Thus Lemma~\ref{lem:transverse} forces vanishing successively
on the larger open sets in~\eqref{eq:open-chain}.
For anisotropic $H$, the boundary consists only of $G_{3111}$
and $G_{2000}$, so this last argument applies directly after
vanishing on $O$.

\paragraph{\textbf{6. From distribution symmetry to multiplicity one.}}
The preceding steps establish transpose symmetry in the
distribution spaces required by the Gelfand--Kazhdan criterion.
Applying the appropriate criterion gives
\[
 \dim\Hom_S(\pi,\eta)\,
 \dim\Hom_{\tau(S)}(\pi^\vee,\eta\circ\alpha)\leq1,
 \qquad \alpha(g)=\tau(g)^{-1}.
\]
To turn this product bound into the desired uniqueness
statement, we show that $\pi^\vee\simeq\pi\circ\alpha$, using
the contragredient identifications in~\cite[Theorem~B]{RV} in the
non-Archimedean case and~\cite[Theorem~1.2 and Corollary~1.4]{Adams}
in the real case. Since
$\alpha(\tau(S))=S$, the two period spaces have the same
dimension. The product bound is thus the square of the desired
multiplicity, which proves Theorem~\ref{thm:main}.

\subsection{Organization of the paper}

The paper follows this separation between common geometry and
the two support arguments. Section~\ref{sec:tools} fixes the
groups, characters, and distribution conventions, states the
support and descent results used in each setting, and identifies
the contragredient. With these tools in place,
Section~\ref{sec:orbits} determines the parabolic double cosets,
proves the open-cell symmetry, and gives the common stabilizer
calculations and normal forms on $G_{3221}$. The resulting calculations are
then used in Section~\ref{sec:padic}, which treats the common
stratum $G_{2000}$, followed by the remaining quasi-split and
non-quasi-split strata, and then deduces the multiplicity bound
from the non-Archimedean distribution theorem.
Section~\ref{sec:real} combines the same geometry with
the Casimir and support arguments to prove the real cases and
explains the role of the full similitude group.

\section{Acknowledgements}
The authors are grateful to their advisor, Prof. Dihua Jiang, for suggesting this problem, for many valuable discussions, and for directing them to relevant references that helped guide its solution.

The authors used ChatGPT Astra and Claude to find references, explore ideas, and obtain limited assistance with \LaTeX{} typesetting and proof checking. Nevertheless, we take full responsibility for all mathematical results, computations, and exposition in this paper.

\section{Notation and background}\label{sec:tools}

We fix one set of block coordinates for both Hermitian forms.
We then state the distribution results used in the proofs. The
non-Archimedean argument uses closed support, localization, and
stabilizers. The Archimedean argument also requires differential
operators to control normal derivatives. The final paragraph
identifies the contragredient, which completes the multiplicity
argument in each case.

\subsection{Notation}\label{sec:model}

Throughout, all unlabelled ranks are over $E$. Write
$E^-=\{z\in E:\bar z=-z\}$ and $N(z)=z\bar z$.
We use $e_{ij}$ for $2\times2$ matrix units and $E_{ij}$ for
$6\times6$ matrix units. Write $a^\dagger=\bar a^t$, let $w_n$ be the anti-diagonal
permutation matrix, and set $J=w_2$. Choose
\[
 H=J\quad\hbox{in the quasi-split case},\qquad
 H=\diag(1,d),\quad-d\notin N_{E/F}(E^\times)
       \quad\hbox{in the anisotropic case}.
\]
Over $\R$ the anisotropic choice will always be $H=I_2$.
For a nondegenerate Hermitian matrix $Q$ of size $n$, write
$\GU(Q)$ for its unitary similitude group, so that
\[
 \GU(Q)(F)=\{g\in\GL_n(E):g^\dagger Qg=\lambda(g)Q
                  \text{ for some }\lambda(g)\in F^\times\}.
\]
The scalar $\lambda(g)$ is called the \emph{similitude factor}
of $g$; the map $\lambda:\GU(Q)(F)\to F^\times$ is the
similitude character. Put
\[
 W=w_3\otimes H,\qquad G=\GU(W)(F),\qquad K=\GU(H)(F).
\]
For $a\in\GL_2(E)$ set $a^*=H^{-1}(a^\dagger)^{-1}H$,
the inverse of its adjoint with respect to $H$.
Put $A_0=\langle e_1,e_2\rangle$ and $B_0=\langle e_5,e_6\rangle$.
The parabolic subgroup stabilizing $A_0$ is $P=MU$. Its Levi
factor consists of the matrices
\[
 m(a,k)=\diag(a,k,\lambda(k)a^*),\qquad
 a\in\GL_2(E),\quad k\in K.
\]
The elements of $U$ are the matrices
\[
 u(X,Y)=\begin{pmatrix}I_2&X&Y\\0&I_2&X'\\0&0&I_2\end{pmatrix},
 \qquad
 \begin{aligned}
 &X'=-H^{-1}X^\dagger H,\\
 &HY+Y^\dagger H+(X')^\dagger HX'=0,
 \end{aligned}
\]
where $X,Y\in M_2(E)$.
Let $L_a=m(a,I_2)$, $D(k)=m(k,k)=\diag(k,k,k)$, and $S=D(K)U$.
Fix a nontrivial unitary additive character $\psi$ of $F$ and smooth
multiplicative characters $\chi_E$ of $E^\times$ and $\chi_F$ of
$F^\times$. Define
\begin{align*}
 \xi(u(X,Y))&=\psi(\Tr(\tr X)),\\
 \omega(k)&=\chi_E(\det k)\chi_F(\lambda(k)),\qquad
 \eta(D(k)u)=\omega(k)\xi(u).
\end{align*}
The trace is invariant under conjugation by $k$, so $\eta$ is a
character of $S$. In the real sections write $\psi(t)=e^{i\kappa t}$,
$\kappa\ne0$.

For the two choices of $H$, define the anti-involution
\begin{equation}\label{eq:tau-common}
 \tau(g)=
 \begin{cases}
 g^t,&H=J,\\[2pt]
 C^{-1}g^tC,&H=\diag(1,d),\quad C=\diag(H,H,H).
 \end{cases}
\end{equation}
On $2\times2$ blocks denote the corresponding operation by $\tau_2$.
These maps preserve $G$ and $K$, respectively. For the weighted
formula this follows from $W=(w_3\otimes I_2)C$; for ordinary
transpose use $W^2=I_6$. They preserve determinants and multipliers.
In both real cases $\tau(g)=g^t$.
Put
\[
 \HH=S\times S,\quad (s_1,s_2)g=s_1g\tau(s_2),\quad
 \chi=\eta\boxtimes\eta,
 \qquad\widetilde\HH=\HH\rtimes\langle\tau\rangle.
\]
The involution swaps the two factors. Extend the character by
$\widetilde\chi(\tau)=-1$. Equivariance is written
\begin{equation}\label{eq:equiv}
                 f(sg)=f(g\tau(s))=\eta(s)f(g).
\end{equation}

We call a $(\widetilde\HH,\widetilde\chi)$-equivariant distribution
\emph{odd}; equivalently it is $(\HH,\chi)$-equivariant and is
negated by $\tau$. Subtracting the $\tau$-transform of an
equivariant distribution produces an odd distribution. Thus
transpose invariance is equivalent to the vanishing of the odd
part of the distribution space under consideration. In the
non-Archimedean argument, this formulation allows
Lemma~\ref{lem:padic-orbits} to apply both to stabilizer elements
in $\HH$ and to elements involving $\tau$.

The stabilizer of $\xi$ in
$M$ is precisely $D(K)$: conjugation sends $X$ to $aXk^{-1}$,
and invariance of $\operatorname{Tr}_{E/F}\tr X$ for all $X$
forces $k^{-1}a=I_2$ by the nondegenerate trace pairing.

Write $\uu=\Lie(U)$. The block action of $m(a,k)$ gives
\begin{equation}\label{eq:unipotent-density}
 \left|\det_F\!\left(\Ad(m(a,k))|_{\uu}\right)\right|_F
      =|N(\det a)|_F^4|\lambda(k)|_F^{-8}.
\end{equation}
For $a=k$, this equals one because $N(\det k)=\lambda(k)^2$,
so $D(K)$ preserves Haar measure on $U$. Since $K$ is reductive
and $U$ is unipotent, both are unimodular; hence so are
$S=D(K)U$ and $\HH=S\times S$.

\subsection{Background over a non-Archimedean local field}
\label{subsec:padic-tools}

In this subsection $F$ is non-Archimedean. We recall Bernstein
localization, Frobenius descent, and restriction to a slice, and
state the Gelfand--Kazhdan criterion used in
Section~\ref{sec:padic}. All characters are smooth.

\subsubsection{Distributions and closed support.}
An $l$-space is a locally compact, totally disconnected Hausdorff
space. Write $\mathcal D(X)=C_c^\infty(X)$ for its locally
constant compactly supported functions and $\mathcal D'(X)$ for
their complex linear dual. For an $l$-group $L$ acting on $X$,
write
\[
 (\ell_*T)(\varphi)=T(\varphi\circ\ell),\qquad
 \mathcal D'_\theta(X)=
 \{T:\ell_*T=\theta(\ell)^{-1}T\text{ for all }\ell\in L\}.
\]
On a manifold with an invariant density this convention agrees
with $f(\ell x)=\theta(\ell)f(x)$ for generalized functions.
We also write $\mathcal D'(X)^{L,\theta}$ for this space.

For a closed subset $Z\subset X$ with complement $V$, restriction
of test functions and extension by zero give
\begin{equation}\label{eq:closed-support-tools}
\begin{gathered}
 0\longrightarrow\mathcal D(V)\longrightarrow\mathcal D(X)
     \longrightarrow\mathcal D(Z)\longrightarrow0,\\
 0\longrightarrow\mathcal D'(Z)\longrightarrow\mathcal D'(X)
     \longrightarrow\mathcal D'(V)\longrightarrow0.
\end{gathered}
\end{equation}
If the subsets are $L$-stable, taking equivariants is left exact.
Consequently vanishing on the successive differences of a finite
closed invariant filtration implies vanishing on $X$.
In particular, pushforward under the closed inclusion $Z\hookrightarrow X$
identifies distributions on $Z$ with distributions on $X$
supported on $Z$.

For later use, $\mathcal D(X\times Y)=\mathcal D(X)\otimes
\mathcal D(Y)$. If $L$ acts trivially on $Y$ and
$\mathcal D'_\theta(X)=0$, evaluating a distribution on products
of test functions gives $\mathcal D'_\theta(X\times Y)=0$.
See~\cite[\S1.22]{BZ} for this identity for test functions.

\subsubsection{Frobenius descent and character obstructions.}
Let an $l$-group $L$ act continuously on an $l$-space $X$, let
$L_0\subset L$ be a closed subgroup, and let $\theta$ be a smooth
character of $L$. Write $\Delta_L$ for the modular character, with
the convention
$\int_L\varphi(x\ell)\,dx=\Delta_L(\ell)^{-1}\int_L\varphi(x)\,dx$
for left Haar measure. Suppose $p:X\to L/L_0$ is continuous and
$L$-equivariant, and put $X_0=p^{-1}(L_0)$. Frobenius descent gives
\begin{equation}\label{eq:frobenius-tools}
 \mathcal D'_\theta(X)\simeq
 \mathcal D'_{\theta|_{L_0}\nu}(X_0),
 \qquad \nu=\Delta_L|_{L_0}\Delta_{L_0}^{-1};
\end{equation}
see~\cite[\S1.5]{Bernstein}. In particular an orbit through $x$
can carry a nonzero equivariant distribution only if
\begin{equation}\label{eq:orbit-modulus}
 \theta^{-1}|_{L_x}=\Delta_L|_{L_x}\Delta_{L_x}^{-1}.
\end{equation}
When this condition holds, the space of equivariant distributions
on the orbit is one-dimensional. Keeping this modulus is essential unless unimodularity has
been verified.

\begin{lemma}[Character obstruction]\label{lem:padic-orbits}
Let $\mathbf L$ be a linear algebraic $F$-group with an
$F$-rational action on an $F$-variety $\mathbf Y$.
Put $L=\mathbf L(F)$, let $X\subset\mathbf Y(F)$ be an
$L$-stable locally closed subset, and let
$\theta:L\to\C^\times$ be a smooth character.
For $x\in X$, write $L_x=\{\ell\in L:\ell x=x\}$.
If $L_x$ contains a unipotent element $u$ with $\theta(u)\ne1$,
or an element $a$ with $\theta(a)=-1$, then
$\mathcal D'_\theta(L\cdot x)=0$.
If one of these conditions holds for every $x\in X$, then
$\mathcal D'_\theta(X)=0$.
\end{lemma}
\begin{proof}
By~\eqref{eq:orbit-modulus}, a nonzero equivariant distribution
on $L\cdot x$ would require
$\theta^{-1}|_{L_x}=\Delta_L|_{L_x}\Delta_{L_x}^{-1}$.
The character on the right takes positive real values and is
trivial on unipotent elements. Either stabilizer condition
therefore contradicts this equality.

By~\cite[\S6.15, Theorem~A]{BZ}, the action of $L$ on
$\mathbf Y(F)$ is constructive, and the same holds on the
invariant locally closed subset $X$.
The result in~\cite[Theorem~6.9]{BZ}, applied to the constant
rank-one sheaf with its equivariant structure twisted by
$\theta^{-1}$, now gives $\mathcal D'_\theta(X)=0$.

\end{proof}

In the boundary calculations, we use Lemma~\ref{lem:padic-orbits}
to deduce vanishing from elements that carry a point to its transpose.
Let $\sigma(s_1,s_2)=(s_2,s_1)$. If $x\in G$ and $t\in\HH$
satisfy
\begin{equation}\label{eq:padic-involutive-transporter}
 t\cdot x=\tau(x),\qquad\chi(t)=1,\qquad\sigma(t)=t^{-1},
\end{equation}
then $a_x=t^{-1}\tau$ fixes $x$, has order two, and satisfies
$\widetilde\chi(a_x)=-1$. Lemma~\ref{lem:padic-orbits}
therefore excludes nonzero odd equivariant distributions on
$\widetilde\HH\cdot x$. At a $\tau$-fixed point, we may take $t=1$.

It suffices to perform this calculation on chosen orbit
representatives. Indeed, if $y=h\cdot x$ with $h\in\HH$
and $a_y\in\widetilde\HH$ fixes $y$, then
$a_x=h^{-1}a_yh$ fixes $x$ and
$\widetilde\chi(a_x)=\widetilde\chi(a_y)$.
Thus the element $h$ used to change representatives need not
have character one.

\subsubsection{Bernstein localization.}
\begin{theorem}[Localization]\label{thm:localization}
Let $p:X\to Y$ be a continuous map of $l$-spaces, and suppose an
$l$-group $L$ acts on $X$ and preserves every fiber. If
$\mathcal D'_\theta(p^{-1}(y))=0$ for every $y\in Y$, then
$\mathcal D'_\theta(X)=0$.
\end{theorem}
This is the localization principle of~\cite[\S1.4]{Bernstein}.
It applies to odd distributions by including the involution in
$L$, provided that the involution preserves $p$.
For example, the determinant on $M_2(E)$ is invariant under
conjugation and transpose. On $F^2$, the product map
$(x,y)\mapsto xy$ is invariant under the action
$t\cdot(x,y)=(tx,t^{-1}y)$ of $F^\times$ and the interchange
$(x,y)\mapsto(y,x)$.
Both maps will occur below. We use localization on fibers, rather
than assume that a distribution on an arbitrary invariant open
set extends equivariantly to its ambient space.

\subsubsection{Restriction to a slice.}
Let $L$ act analytically on a manifold $X$. A submanifold $M_0$
is an $L$-slice if
\[
 \rho:L\times M_0\longrightarrow X,\qquad(\ell,m)\longmapsto\ell m
\]
is a surjective submersion. For generalized functions, which act
on locally constant test densities, pullback by $\rho$ is defined
and injective. Equivariance factors the pullback into a character
on $L$ and a generalized function on $M_0$. A subgroup preserving
$M_0$ imposes the corresponding restricted equivariance. This is
the slice restriction of~\cite[Lemma~2.13]{JLX}.

To express it for distributions on test functions, one must choose
densities and account for their transformation characters. For the
open cell, Proposition~\ref{prop:open} uses bi-invariant Haar
densities on $G$ and $\GL_2(E)$, which are also preserved by
$\tau$ and $\tau_2$, respectively. On the plane
family in Section~\ref{subsec:plane-padic}, we give the full
fiber action and check its density explicitly, including at the
origin. These checks show that no additional character occurs
when we restrict to these slices.
If a submanifold of the source of a surjective submersion is a
union of fibers, its image is a submanifold; if it is closed, its
image is closed. This follows in local product coordinates and
yields the statement for orbit saturations in~\cite[Lemma~2.12]{JLX},
which we use in Section~\ref{sec:orbits}.

\subsubsection{The Gelfand--Kazhdan criterion.}
Let $G_0$ be a unimodular $l$-group, let $S_1,S_2$ be closed
subgroups, and let $\eta_i:S_i\to\C^\times$ be smooth characters.
Fix a continuous anti-involution $\sigma$ of $G_0$. Suppose every
generalized function on $G_0$ satisfying
\[
 f(s_1g)=\eta_1(s_1)f(g),\qquad
 f(gs_2)=\eta_2(s_2)^{-1}f(g)
\]
is $\sigma$-invariant. Then the Gelfand--Kazhdan criterion gives,
for every irreducible admissible smooth complex representation
$\pi$ of $G_0$,
\begin{equation}\label{eq:GK-product-general}
 \dim\Hom_{S_1}(\pi,\eta_1)\,
 \dim\Hom_{S_2}(\pi^\vee,\eta_2)\leq1.
\end{equation}
By~\cite[Theorem 2.21]{JLX}, it is enough to prove the required
symmetry for distributions on which the Bernstein center acts
through a character. Our hypothesis gives this symmetry for every
distribution with the stated equivariance, so that theorem applies.
In our application
\begin{equation}\label{eq:GK-characters}
 S_1=S,\quad S_2=\tau(S),\quad\eta_1=\eta,\quad
 \eta_2(b)=\eta(\tau(b)^{-1}).
\end{equation}
Thus the two equations for generalized functions are exactly
\eqref{eq:equiv}. The contragredient identification in
Lemma~\ref{lem:dual} below identifies the two factors in
\eqref{eq:GK-product-general}.

\subsection{Background over an Archimedean local field}
\label{subsec:real-tools}

Here $E/F=\C/\R$ and $H=J$ or $I_2$. We begin with the
Gelfand--Kazhdan criterion, which reduces the multiplicity problem
to symmetry of tempered generalized functions. We then introduce
the function spaces and the restriction and localization results
used on the open cell in Section~\ref{sec:orbits}, followed by
the differential operators and support lemmas needed for the
boundary argument in Section~\ref{sec:real}. We conclude by
identifying the contragredient representations needed to deduce
multiplicity one from the criterion.

\subsubsection{The Archimedean Gelfand--Kazhdan criterion.}
\label{subsubsec:real-GK}
For a real reductive group $G_0$, put
$\mathfrak g_{0,\C}=\Lie(G_0)\otimes_\R\C$ and
\[
 \mathcal Z_{G_0}=U(\mathfrak g_{0,\C})^{G_0},
\]
where $U(\mathfrak g_{0,\C})$ is the universal enveloping algebra.
We identify $\mathcal Z_{G_0}$ with the algebra of bi-invariant
differential operators on $G_0$. Write $\ctemp(G_0)$ for the
space of tempered generalized functions, namely continuous
linear functionals on Schwartz densities; the precise definitions
are given in Section~\ref{subsubsec:nash-functions}.

\begin{proposition}[Gelfand--Kazhdan criterion]\label{prop:real-GK}
Let $S_1,S_2$ be closed subgroups of a real reductive group $G_0$,
let $\eta_i:S_i\to\C^\times$ be continuous characters, and let
$\sigma$ be a continuous anti-automorphism of $G_0$. Suppose
every $f\in\ctemp(G_0)$ that is a joint eigenvector for
$\mathcal Z_{G_0}$ and satisfies
\[
 f(s_1gs_2)=\eta_1(s_1)\eta_2(s_2)^{-1}f(g),
 \qquad s_i\in S_i,\quad g\in G_0,
\]
also satisfies $f\circ\sigma=f$. Then every irreducible
Casselman--Wallach representation $\pi$ of $G_0$ satisfies
\[
 \dim\Hom_{S_1}(\pi,\eta_1)\,
 \dim\Hom_{S_2}(\pi^\vee,\eta_2)\leq1.
\]
Here $\pi^\vee$ is the smooth contragredient, the intertwining
maps are continuous, and the identities for $f$ are interpreted
in the sense of generalized functions.
\end{proposition}
This is~\cite[Proposition~10.1]{JSZ}.

For our model, take $G_0=G$, $\sigma=\tau$, and the subgroups
and characters in~\eqref{eq:GK-characters}. The equivariance
condition is then exactly~\eqref{eq:equiv}. The Casimir operator
$\Omega$, defined in Section~\ref{subsubsec:real-lie}, belongs
to $\mathcal Z_G$, so every joint eigenvector in the criterion
satisfies $\Omega f=cf$ for some $c\in\C$.
Theorem~\ref{thm:distribution} therefore supplies the required
symmetry. Lemma~\ref{lem:dual} identifies the two Hom spaces
in the resulting product bound and yields Theorem~\ref{thm:main}.

\subsubsection{Nash manifolds and generalized functions.}\label{subsubsec:nash-functions}
We now describe the generalized functions appearing in the
criterion. Their definition uses Nash manifolds and Schwartz
sections, which also allow the restriction to slices and
localization used in the real open-cell argument of
Proposition~\ref{prop:open}.

A subset of $\R^n$ is \emph{semialgebraic} if it is a finite
Boolean combination of sets defined by polynomial equalities and
inequalities. A \emph{Nash manifold} is a smooth manifold with a
finite atlas whose coordinate images are open semialgebraic sets
and whose transition maps have semialgebraic graphs. A \emph{Nash
map} is a smooth map whose graph is semialgebraic in these charts;
a real-valued Nash map is a \emph{Nash function}. A \emph{Nash open
subset} is an open subset that is semialgebraic in the charts.
The real matrix groups and smooth semialgebraic matrix strata used
here have their natural Nash structures. Their algebraic actions
and the map $\det_\C:M_2(\C)\to\C$, viewed over $\R$, are Nash.

For an open semialgebraic set $V\subset\R^n$, the Schwartz space
$\mathcal S(V)$ consists of the smooth complex-valued functions
$\varphi$ satisfying
\[
 p_{a,\alpha}(\varphi):=
 \sup_{x\in V}|a(x)\partial^\alpha\varphi(x)|<\infty
\]
for every Nash function $a$ on $V$ and every multi-index $\alpha$,
with the topology defined by these seminorms. For $V=\R^n$ this
recovers the usual Schwartz space. In general, extension by zero
identifies $\mathcal S(V)$ with the Schwartz functions on $\R^n$
whose derivatives of every order vanish on $\R^n\setminus V$
\cite[Definition~4.1.1 and Theorem~5.4.1]{AGSchwartz}.

To define densities intrinsically, let $\operatorname{Dens}_X\to X$
be the complex line bundle of densities on an $n$-dimensional
smooth manifold $X$. Its local frames $|dx_1\cdots dx_n|$ satisfy
\[
 |dx_1\cdots dx_n|
 =\left|\det\left(\frac{\partial x}{\partial y}\right)\right|
   |dy_1\cdots dy_n|
\]
under a change of coordinates. Smooth densities are smooth
sections of this line bundle, and
$C_c^\infty(X,\operatorname{Dens}_X)$ denotes its compactly supported
smooth sections; see~\cite[Appendix~A.1.1]{AGSchwartz}.

If $X$ is Nash, choose a finite Nash atlas $x_i:U_i\to V_i\subset\R^n$.
We define the space of \emph{Schwartz sections}
\[
 \mathcal S(X,\operatorname{Dens}_X)
 \subset C^\infty(X,\operatorname{Dens}_X)
\]
to consist of sums of sections obtained by extending
$(\varphi_i\circ x_i)|dx_{i,1}\cdots dx_{i,n}|$ by zero from $U_i$
to $X$, where $\varphi_i\in\mathcal S(V_i)$. These sections are
also called \emph{Schwartz densities}. We give this space the
quotient topology induced by the resulting surjection
$\bigoplus_i\mathcal S(V_i)\to\mathcal S(X,\operatorname{Dens}_X)$.
Both the space and its topology are independent of the atlas
\cite[Definition~5.1.2 and Proposition~5.1.2]{AGSchwartz}.

We define the spaces of generalized functions by
\[
 \begin{aligned}
 \cgen(X)&=C_c^\infty(X,\operatorname{Dens}_X)',\\
 \ctemp(X)&=\mathcal S(X,\operatorname{Dens}_X)'
             \subset\cgen(X),
 \end{aligned}
\]
where a prime denotes the continuous dual with its strong topology.
The first space is defined on every smooth manifold; the second
is the space of tempered generalized functions on a Nash manifold.
The superscript $-\xi$ in $\ctemp$ follows the convention of \cite[\S2.3]{JSZ}
and is unrelated to the character $\xi$ of $U$.

If a Lie group $L$ acts smoothly on $X$ and $\theta:L\to\C^\times$
is a continuous character, write $\cgen_\theta(X)$ for the
subspace satisfying $f(\ell x)=\theta(\ell)f(x)$ in the sense
of generalized functions. When $X$ and the action are Nash,
put $\ctemp_\theta(X)=\ctemp(X)\cap\cgen_\theta(X)$.

\subsubsection{Slices and tempered localization.}\label{subsubsec:real-localization}
With these function spaces fixed, we recall the restriction and
localization results used in Proposition~\ref{prop:open}. Let
$L$ act smoothly on $X$, let $\theta$ be a continuous character
of $L$, and let $M_0\subset X$ be a smooth submanifold.
If the action map $L\times M_0\to X$ is a surjective submersion,
pullback and removal of the character factor give an injective
restriction map for equivariant generalized functions. When the
action and the slice are Nash, this restriction preserves
temperedness. These are
\cite[Lemma 2.7 and the paragraph following Lemma 2.8]{JSZ}.
This restriction is defined by the submersion; it is not an
arbitrary pullback to an embedded submanifold.

We also use the following particular localization statement.
If $p:X\to Y$ is an $L$-invariant Nash map and
$\ctemp_\theta(X)=0$, then
\begin{equation}\label{eq:real-localization}
 \ctemp_\theta(p^{-1}(Y_0))=0
 \quad\text{for every Nash open }Y_0\subset Y.
\end{equation}
This is~\cite[Lemma 2.8]{JSZ}. It will be applied to
$p=\det_E:M_2(E)\to E$ and $Y_0=E^\times$.

In the hyperbolic case, the resulting matrix problem also uses
the following result about the orthogonal group:
for a nondegenerate real quadratic space $V_0$ and $r<\dim V_0$,
an $\SO(V_0)$-invariant tempered generalized function on $V_0^r$
is $\OO(V_0)$-invariant; see~\cite[Proposition 1.5]{JSZ}.
Our application has $V_0=\mathfrak{sl}_2(\R)$,
$q(A)=\tr(A^2)$, and $r=2$. Scalar matrix coordinates are
fixed and can be included as extra parameters.

\subsubsection{Lie algebras and the Casimir operator.}
\label{subsubsec:real-lie}
To pass from the open cell to the boundary, we need to take
account of the difference between the two local-field settings.
Over a non-Archimedean
field, \eqref{eq:closed-support-tools} identifies distributions
supported on a closed subset with distributions on that subset;
but over $\R$, a distribution supported on a smooth boundary piece
may also involve normal derivatives, so orbitwise vanishing alone
is insufficient. The generalized matrix coefficients in the
Gelfand--Kazhdan criterion satisfy the Casimir equation
$\Omega f=cf$. To use this equation, the Lie algebra calculation
in Section~\ref{subsec:casimir} differentiates equivariance and
yields a transverse derivative term which cannot be cancelled
by the tangential terms (Lemma~\ref{lem:transverse}). For the
remaining quasi-split stratum $G_{3221}$, we also need the support
results for unipotent stabilizers and metrically proper
submanifolds recalled below from~\cite[\S\S2--3]{JSZ}.

All Lie algebras in this paragraph are real. The Lie algebra of $G$ is
\begin{equation}\label{eq:lie-algebra}
 \g=\{Z\in M_6(\C):Z^\dagger W+WZ=\ell(Z)W,
                    \ \ell(Z)\in\R\}.
\end{equation}
Write $\kk=\Lie(K)$, $\mm=\Lie(M)$ and $\uu=\Lie(U)$. Put $\s=\Lie(S)$.
For $Z\in\kk$, the scalar $\ell(Z)\in\R$ is defined by
$Z^\dagger H+HZ=\ell(Z)H$.
The differential of the diagonal embedding $D$ is $D_*(Z)=\diag(Z,Z,Z)$, so
$\s=D_*(\kk)\ltimes\uu$.
In particular,
\[
 \mm=\left\{\diag(A,Z,\ell(Z)I_2-H^{-1}A^\dagger H):
              A\in M_2(\C),\ Z\in\kk\right\},
\]
and
\[
 \uu=\left\{N(A,C_0)=
 \begin{pmatrix}0&A&C_0\\0&0&-H^{-1}A^\dagger H\\0&0&0\end{pmatrix}:
             HC_0+C_0^\dagger H=0\right\}.
\]
The notation $\mm_\C=\mm\otimes_\R\C$ means complexification;
$U(\mm_\C)$ denotes its universal enveloping algebra. We use the same convention
for $\g_\C$ and $U(\g_\C)$.

For $Z\in\g$, the left-invariant differential operator is
\[
 Z^Lf(g)=\left.\frac{d}{dt}\right|_{t=0}f(g\exp(tZ)).
\]
Thus its tangent vector at $g$ is $gZ$. Extend this action
complex-linearly and multiplicatively to $U(\g_\C)$, and by
duality to generalized functions. This convention determines
the signs in the Casimir calculation.

Use the real bilinear form and the associated metric
\begin{equation}\label{eq:metric}
 B(Z,Z')=\Ree\tr(ZZ'),\qquad
 \langle V,V'\rangle_g=B(g^{-1}V,g^{-1}V').
\end{equation}
It is nondegenerate: $\g=\mathfrak u(W)\oplus\R I_6$ is an
orthogonal decomposition, the second summand has norm $6$, and
the trace form on the first is nondegenerate by complexification.
It is $\Ad(G)$-invariant, so the metric is bi-invariant.

Choose a real basis $(Z_j)$ of $\g$, and let $(Z^j)$ be its
$B$-dual basis, characterized by $B(Z_i,Z^j)=\delta_{ij}$.
The Casimir element associated with $B$ is
\[
 \Omega=\sum_j Z_jZ^j\in U(\g_\C),
\]
where the products are taken in the universal enveloping algebra.
This element is independent of the chosen basis, and the
$\Ad(G)$-invariance of $B$ implies that it is central.
Under the differential action defined above, the same symbol
denotes the second-order operator
\[
 \Omega f=\sum_j Z_j^L\bigl((Z^j)^L f\bigr),
\]
extended by duality to generalized functions. This operator is
bi-invariant, has the inverse metric as its principal symbol,
and commutes with $\tau$. The same construction with the
nondegenerate restriction $B|_{\mm}$ defines $\Omega_M$.
The decomposition
\[
 \g=\uu\oplus\mm\oplus\uu^t
\]
has isotropic summands $\uu$ and $\uu^t$, which are paired
nondegenerately by $B$. The summand $\mm$ is orthogonal to both. Section~\ref{subsec:casimir} uses this
decomposition and right $U^t$-equivariance to express $\Omega$
as a transverse vector field plus an operator from $U(\mm_\C)$.

\subsubsection{Generalized functions supported on submanifolds.}
\label{subsubsec:real-support}
After vanishing on the open cell has been proved, the remaining
task is to exclude generalized functions supported on the boundary.
The lemmas below allow us to use the Casimir equation and stabilizer
actions for this purpose in Section~\ref{sec:real}, while accounting
for normal derivatives of every order. These local statements
apply to all generalized functions on smooth manifolds. For a
closed submanifold $Z\subset X$, write $\cgen(X;Z)$ for the
generalized functions on $X$ supported on $Z$.

\begin{lemma}[Transverse differential operator]\label{lem:transverse}
Let $Z$ be a smooth closed submanifold of $X$. Suppose $V$ is
a smooth vector field with $V_z\notin T_zZ$ for every $z\in Z$,
and $D$ is tangential to $Z$. Then
\[
 f\in\cgen(X;Z),\qquad(V+D)f=0\quad\Longrightarrow\quad f=0.
\]
Here tangential means locally a sum of products of vector fields
tangent to $Z$, with smooth coefficients; $D$ may have any finite
order.
\end{lemma}
This is the first-order case of~\cite[Lemma 2.2]{JSZ}. One may
see the mechanism in normal coordinates: $V$ raises the largest
normal order with nonzero leading coefficient, whereas a
tangential operator does not.

The Casimir equation also gives a vanishing criterion through the
geometry of the trace metric. Metrical properness, as defined below, ensures that the
principal symbol of $\Omega-c$ has a nonzero component in the normal
directions, so Lemma~\ref{lem:metric-support} excludes nonzero
solutions supported on the submanifold. We use this criterion on
$C\setminus C_0$ and $Z$ in the quasi-split case
(Lemmas~\ref{lem:Cmetric} and~\ref{lem:Zmetric}).

\begin{definition}\label{def:metric-proper}
A submanifold $Z$ of a manifold $X$ with a nondegenerate metric is
\emph{metrically proper} if, at each $z$, $T_zZ$ is contained in
a proper nondegenerate subspace of $T_zX$.
\end{definition}
It is enough to exhibit $v\perp T_zZ$ with $\langle v,v\rangle\ne0$:
then $v^\perp$ is such a subspace.

A second-order differential operator $D$ on $X$ is \emph{of
Laplacian type} with respect to this metric if its principal
symbol is the inverse metric; see~\cite[Definition~3.1(b)]{JSZ}.
Equivalently, in local coordinates $(x_1,\ldots,x_n)$,
\[
 D=\sum_{i,j=1}^n b^{ij}(x)
       \frac{\partial^2}{\partial x_i\,\partial x_j}+D_1,
 \qquad (b^{ij})=(b_{ij})^{-1},
\]
where $b_{ij}=\langle\partial_{x_i},\partial_{x_j}\rangle$
and $D_1$ has order at most one. Its principal symbol has a
nonzero normal component along every metrically proper
submanifold, giving the following vanishing criterion.

\begin{lemma}[Metrically proper support]\label{lem:metric-support}
Let $Z$ be metrically proper and closed in an open set $V\subset X$.
If $D$ is of Laplacian type, then
\[
 f\in\cgen(V;Z),\quad Df=0\quad\Longrightarrow\quad f=0.
\]
In particular this applies to $D=\Omega-c$ on $G$.
\end{lemma}
This is~\cite[Lemma 3.2]{JSZ}. For the trace metric~\eqref{eq:metric},
$\Omega-c$ is of Laplacian type because $\Omega$ has the inverse
metric as its principal symbol, and subtracting $c$ does not
change that symbol.

\begin{definition}\label{def:unipotent-incompatible}
Let a Lie group $L$ act smoothly on $X$, and let $\theta$ be a
character. An invariant submanifold $Z$ is \emph{unipotently
$\theta$-incompatible} if, near every point, there is a local
$L$-slice $Z_0$ in $Z$ and a smooth map $a:Z_0\to L$ such that
$a(z)z=z$, $\theta(a(z))\ne1$, and $a(z)$ acts unipotently
on $T_zX/T_zZ$.
\end{definition}
For an algebraic action, a smooth family of unipotent stabilizing
elements with nontrivial character has these properties. This is
the form used in our matrix calculations.

\begin{lemma}[Unipotent character obstruction]\label{lem:real-unipotent}
If $Z$ is closed in an invariant open set $V$ and is unipotently
$\theta$-incompatible in $X$, then
$\cgen_\theta(V;Z)=0$.
\end{lemma}
This is~\cite[Lemma 3.4]{JSZ}. On the terms of a fixed normal
order, the normal action and its symmetric powers are unipotent.
Multiplication by the nontrivial scalar $\theta(a(z))$ makes
the equivariance equation invertible. Local slice restriction
justifies this reasoning for generalized sections.

\begin{lemma}[Finite filtration]\label{lem:real-filtration}
Let $Z$ be a closed invariant subset of a manifold with metric.
Suppose $Z$ has a finite closed invariant filtration whose
successive differences are smooth and are either metrically
proper or unipotently $\theta$-incompatible in the ambient
manifold. If $D$ is of Laplacian type, then
$\cgen_\theta(X;Z)\cap\ker D=0$.
\end{lemma}
This is~\cite[Lemma 3.7]{JSZ}; that paper calls the hypothesis
the $U_\theta M$ property. We will use the stated filtration
directly. Induction removes its successive open differences by
the two preceding lemmas.

\subsubsection{Contragredients in both settings.}
The Gelfand--Kazhdan criterion gives a product of two
multiplicities. The following description of the contragredient
identifies the two factors in our application. We record both
local-field cases for use in Sections~\ref{sec:padic}
and~\ref{sec:real}.

\begin{lemma}\label{lem:dual}
Put $\alpha(g)=\tau(g)^{-1}$. For every irreducible representation
in the category of Theorem~\ref{thm:main},
$\pi^\vee\simeq\pi\circ\alpha$.
\end{lemma}
\begin{proof}
Over $\R$, $H=J$ or $I_2$, and $\alpha(g)=g^{-t}$.
Choose real orthogonal $R_0$ with
$R_0^tWR_0=\diag(I_p,-I_q)$, where $(p,q)=(3,3)$ or $(4,2)$.
The Cartan subgroup with real points
\[
 \{R_0\diag(z_1,\ldots,z_6)R_0^t:
     z_j\in\C^\times,\ |z_1|=\cdots=|z_6|\}
\]
is compact modulo the one-dimensional split center, and hence
fundamental. Its matrices are symmetric, so $\alpha$ restricts
to inversion. The real Chevalley involution theorem
\cite[Theorem 1.2 and Corollary 1.4]{Adams} gives the asserted
contragredient identification. Its connectedness hypothesis is on
the reductive algebraic group, not on the group of real points;
it therefore applies to the full $\GU(3,3)(\R)$ as well.

Over a non-Archimedean field define an antiunitary involution on
$E^6$ by $h(v)=W\bar v$ in the quasi-split case and
$h(v)=(w_3\otimes I_2)\bar v$ in the anisotropic coordinates.
The defining similitude equation gives, respectively,
\[
 g^{-t}=\lambda(g)^{-1}W\bar gW,
 \qquad
 C^{-1}g^{-t}C=\lambda(g)^{-1}(w_3\otimes I_2)\bar g(w_3\otimes I_2).
\]
Thus $\alpha(g)=\lambda(g)^{-1}hgh^{-1}$.
The dualizing theorem of~\cite[Theorem B]{RV} applies
and yields $\pi^\vee\simeq\pi\circ\alpha$.
This theorem also includes residue characteristic two.
\end{proof}

\section{Orbit decomposition}\label{sec:orbits}

This section establishes the geometric decomposition and the vanishing of an equivariant distribution on the open cell that is needed in both local-field cases of
Theorem~\ref{thm:distribution}. We first classify the smooth
$P$--$\tau(P)$ double cosets and arrange them in a filtration
along which we will prove vanishing. Proposition~\ref{prop:open}
then shows that every odd equivariant distribution on the open
cell vanishes; over $\R$, the assertion concerns tempered
generalized functions. It therefore remains to exclude support
on the boundary in Sections~\ref{sec:padic} and~\ref{sec:real} for $F$ being both non-archimedean and archimedean respectively.

To prepare those arguments, we study the quasi-split boundary
stratum $G_{3221}$ in more detail. A parabolic double coset may
contain many $\HH$-orbits, so we compute their stabilizers and
the restrictions of $\chi$, and construct a closed filtration
with explicit normal forms. We will need these results for proving the vanishing statements for the boundary orbits in Section \ref{sec:padic} and \ref{sec:real} as well.

\subsection{The parabolic double coset decomposition}
Recall from Section~\ref{sec:model} that
$A_0=\langle e_1,e_2\rangle_E$ and
$B_0=\langle e_5,e_6\rangle_E$ are isotropic subspaces of
dimension two, with stabilizers $P$ and $\tau(P)$, respectively.
We classify the $P-\tau(P)$ double cosets in $G$ by the relative position of $A_0$
and $gB_0$. In matrix coordinates, the required invariants are
the ranks
\[
 R(g)=\begin{pmatrix}
 \rank g_{\{3,4,5,6\},\{3,4,5,6\}}&
 \rank g_{\{3,4,5,6\},\{5,6\}}\\
 \rank g_{\{5,6\},\{3,4,5,6\}}&
 \rank g_{\{5,6\},\{5,6\}}
 \end{pmatrix}=\begin{pmatrix}r_{11}&r_{12}\\r_{21}&r_{22}\end{pmatrix}.
\]
Here $g_{I,J}$ denotes the submatrix with row indices $I$ and
column indices $J$. We abbreviate the four entries by $4222$,
$3221$, and so forth, and write $G_{r_{11}r_{12}r_{21}r_{22}}$
for the corresponding orbit inside $G$.
\begin{lemma}\label{lem:rank-common}
The $P$--$\tau(P)$ double cosets are classified by $R(g)$.
For hyperbolic $H$ their ranks are $4222,3221,3111,2110,2000$;
for anisotropic $H$ they are $4222,3111,2000$. Every stratum is
smooth and stable under $\widetilde\HH$. The open stratum is
$O=G_{4222}=UM\tau(U)$.
\end{lemma}
\begin{proof}
Put $V=gB_0\in G/\tau(P)$. The ranks in $R(g)$ can be expressed in terms of
the two invariants
\[
 t=\dim_E(A_0\cap V),\qquad
 s=2-\dim_E(V\cap A_0^\perp),
\]
where orthogonality is taken with respect to
$\langle x,y\rangle=x^\dagger Wy$.

The coordinate projection $q:E^6\to E^2$,
$q(x_1,\ldots,x_6)=(x_5,x_6)$, has kernel $A_0^\perp$.
In the basis $(ge_5,ge_6)$ of $V$ and the standard basis of
$E^2$, its restriction $q|_V$ has matrix
$g_{\{5,6\},\{5,6\}}$, so $r_{22}=s$.
The restriction of the Hermitian form to $A_0\times V$ has
matrix $Hg_{\{5,6\},\{5,6\}}$, of rank $s$. Hermitian symmetry
therefore gives $\dim_E(A_0\cap V^\perp)=2-s$.

To compute the remaining ranks, let $p:E^6\to E^6/A_0$ be
the quotient map. Since
$g\langle e_3,e_4,e_5,e_6\rangle_E=V^\perp$, the submatrices
of $g$ defining $r_{12}$, $r_{21}$, and $r_{11}$ represent,
respectively, $p|_V$, $q|_{V^\perp}$, and $p|_{V^\perp}$.
Their kernels are $A_0\cap V$,
$V^\perp\cap A_0^\perp=(A_0+V)^\perp$, and
$A_0\cap V^\perp$, of dimensions $t$, $2+t$, and $2-s$.
Rank--nullity therefore gives
\[
 R(g)=\begin{pmatrix}2+s&2-t\\2-t&s\end{pmatrix}.
\]

The Witt index determines the possible values of $(s,t)$. The isotropic subspace $A_0+(V\cap A_0^\perp)$ has
dimension $4-s-t$, while
$A_0\cap V\subset V\cap A_0^\perp$ gives $s+t\leq2$.
If $H$ is hyperbolic, the ambient Witt index is three, hence
$1\leq s+t\leq2$. The possible pairs are
$(2,0),(1,0),(1,1),(0,1),(0,2)$, giving the five displayed
ranks. If $H$ is anisotropic, the Witt index is two, forcing
$s+t=2$ and leaving the three pairs $(2,0),(1,1),(0,2)$.
The representatives below realize every listed case.

These invariants also determine the double coset.
Since $\tau(P)$ stabilizes $B_0$, it suffices to show that two
subspaces $V_1,V_2$ with the same $(s,t)$ are conjugate under
$P$. For each pair $(A_0,V_i)$, choose a basis of $A_0\cap V_i$,
extend it to bases of $A_0\cap V_i^\perp$ and
$V_i\cap A_0^\perp$, and complete these to bases of $A_0$
and $V_i$ by choosing the remaining vectors to be dual under
the Hermitian form. Since both subspaces are isotropic, the
resulting bases of $A_0+V_1$ and $A_0+V_2$ have identical
Hermitian Gram matrices. Matching them gives an isometry carrying
$A_0$ to itself and $V_1$ to $V_2$. Witt extension gives an
isometry of $E^6$ with the same properties, hence an element of
$P$. This is an isometry of the Hermitian spaces over $E/F$,
so the argument classifies double cosets in $G=\GU(W)(F)$.

Each stratum is a homogeneous orbit of $P\times\tau(P)$,
hence smooth, and is $\HH$-stable because $S\subset P$.
Both transpose and weighted transpose send $R(g)$ to $R(g)^t$;
the formula above has equal off-diagonal entries. Thus every
stratum is also $\tau$-stable, and hence $\widetilde\HH$-stable.

Finally, the condition $r_{22}=2$ is open and, by the rank formula,
equivalent to $R(g)=4222$. Since $R(I_6)=4222$, the double
coset classification gives $G_{4222}=P\tau(P)$. 

\end{proof}

Representatives for the quasi-split double cosets are the
following permutation matrices, listed in the order of
Lemma~\ref{lem:rank-common}:
\[
 I_6,\quad(1\ 4)(3\ 6),\quad(2\ 5),\quad
 (1\ 3)(2\ 5)(4\ 6),\quad j=(1\ 5)(2\ 6).
\]
In the anisotropic case they are $I_6,(2\ 6),j$. Each preserves
$W$, and its ranks are obtained directly from the indicated
permutation.

For later calculations, each double coset admits a parametrization
by two $\GL_2(E)$ factors and the $\HH$-action. For every
representative $w$ above, the map
\begin{equation}\label{eq:rank-cover}
 \HH\times\GL_2(E)^2\longrightarrow G_{R(w)},\qquad
 (h,A,B)\longmapsto h\cdot L_AwL_B
\end{equation}
is a surjective submersion. Indeed, the decompositions
$m(a,k)=D(k)L_{k^{-1}a}=L_{ak^{-1}}D(k)$ absorb the two
diagonal $K$ factors and the unipotent factors into $\HH$.
This smooth change of parameters identifies the
$P\times\tau(P)$ orbit map with~\eqref{eq:rank-cover}, so the
submersion assertion follows from that for the orbit map.

The rank conditions also give the filtration used in the
boundary arguments of Sections~\ref{sec:padic} and~\ref{sec:real}.
When $H$ is hyperbolic, the open chain is
\begin{equation}\label{eq:open-chain}
 O\subset G'=\{r_{12}=2\}\subset G_1=\{r_{22}\ge1\}
       \subset G_2=\{r_{12}\ge1\}\subset G.
\end{equation}
The rank conditions define open subsets, and the list in
Lemma~\ref{lem:rank-common} shows that the successive differences
are $G_{3221},G_{3111}, G_{2110},G_{2000}$, each closed in the new open ambient
set. In the anisotropic case the chain is
$O\subset\{r_{22}\ge1\}\subset G$, with successive differences being $G_{3111},G_{2000}$.

\subsection{Distributions on the open cell}\label{sec:open}
The main result of this subsection equates checking $\tau$-invariance of $(
\tilde{\HH},\tilde{\chi})$-equivariant distributions on the open cell $O$ with checking transpose invariance of $K$-equivariant distributions on $\mathrm{GL}_2(E)$, with distributions replaced by tempered generalized functions when $F=\mathbb{R}$.
The next two lemmas establish
this invariance for the two Hermitian planes;
Proposition~\ref{prop:open} then gives the reduction from $O$.

\begin{lemma}\label{lem:linear}
Suppose $H=J$. In the real case, let $T$ be a tempered generalized
function on $\GL_2(\C)$. In the non-Archimedean case, let $T$ be
a distribution on $\GL_2(E)$. If $T$ is invariant under conjugation
by $K$, then it is invariant under transpose.
\end{lemma}
\begin{proof}
A change of basis identifies the conjugation action of $K$
with that of $\GL_2(F)$. Indeed, choose $\delta\in E^\times$
with $\bar\delta=-\delta$, and set
\[
 P_0=\diag(1,\delta),\quad
 J_0=\begin{pmatrix}0&1\\-1&0\end{pmatrix},\quad
 Q_0=P_0^{-2}=\diag(1,\delta^{-2})\in\GL_2(F).
\]
Since $P_0^\dagger JP_0=\delta J_0$, the similitude identity for
$C_0=P_0^{-1}kP_0$ and the identity
$C_0^tJ_0C_0=(\det C_0)J_0$ give
$\bar C_0=(\lambda(k)/\det C_0)C_0$. The scalar on the right
has norm one. Hilbert's theorem 90 therefore gives $z_0\in E^\times$
such that $\bar z_0/z_0=\lambda(k)/\det C_0$, and
$z_0^{-1}C_0\in\GL_2(F)$. The converse inclusion follows from
the same determinant identity. Thus,
\begin{equation}\label{eq:Kfactor}
                         K=E^\times P_0\GL_2(F)P_0^{-1}.
\end{equation}

Under this change of basis, write
$z=P_0(A+\delta B)P_0^{-1}$ and set
\[
 X=M_2(F)^2,\qquad Y=\GL_2(F)^2,\qquad
 X_E^\times=\{(A,B)\in X:\det(A+\delta B)\ne0\}.
\]
This change of coordinates identifies $\GL_2(E)$ with $X_E^\times$;
denote the pullback of $T$ to $X_E^\times$ by $T_0$. By~\eqref{eq:Kfactor},
$T_0$ is invariant under
$(A,B)\mapsto(gAg^{-1},gBg^{-1})$, $g\in\GL_2(F)$.
Since $P_0^t=P_0$, we have
\[
 P_0^{-1}z^tP_0
 =P_0^{-2}(A^t+\delta B^t)P_0^2
 =Q_0(A^t+\delta B^t)Q_0^{-1}.
\]
Thus transpose on $\GL_2(E)$ induces simultaneous transpose
followed by simultaneous conjugation by $Q_0$. As
$Q_0\in\GL_2(F)$, it remains to prove that $T_0$ is invariant
under $\tau_0(A,B)=(A^t,B^t)$. In both field settings, this
will follow by localization from the corresponding invariance
on $X$, with tempered generalized functions understood over $\R$.

For $F=\R$, take $\delta=i$ and separate the scalar and traceless
coordinates:
\[
 A=aI_2+A_0,\qquad B=bI_2+B_0,\qquad
 a=\tfrac12\tr A,\quad b=\tfrac12\tr B,\qquad
 A_0,B_0\in V:=\mathfrak{sl}_2(\R).
\]
This identifies $X$ with $\R^2\times V^2$. Equip $V$ with
$q(v)=\tr(v^2)$. The conjugation action
$g\cdot v=gvg^{-1}$ of the full $\GL_2(\R)$ has image
$\SO(V,q)$; see~\cite[proof of Proposition~7.1]{JSZ}.
In the basis
\[
 \begin{pmatrix}1&0\\0&-1\end{pmatrix},\qquad
 \begin{pmatrix}0&1\\1&0\end{pmatrix},\qquad
 J_0,
\]
the form $q$ has matrix $\diag(2,2,-2)$ and transpose has matrix
$\diag(1,1,-1)$. Hence $q$ is nondegenerate, and transpose
belongs to $\OO(V,q)\setminus\SO(V,q)$.
The result in~\cite[Proposition~1.5]{JSZ} states that, for a
nondegenerate real quadratic space $V_0$ and $r<\dim V_0$,
\[
 \ctemp(V_0^r)^{\SO(V_0)}=\ctemp(V_0^r)^{\OO(V_0)}.
\]
Since $\dim V=3$, the result applies to $V^2$.
To include the scalar coordinates, let $\OO(V,q)$ act diagonally
on $V^2$ and trivially on $\R^2$, and put
$\varepsilon(g)=\det_V(g)$.
For $f\in\ctemp_\varepsilon(X)$ and a Schwartz density
$\rho\in\mathcal S(\R^2,\operatorname{Dens}_{\R^2})$, define
\[
 f_\rho(\varphi)=f(\rho\boxtimes\varphi),\qquad
 \varphi\in\mathcal S(V^2,\operatorname{Dens}_{V^2}).
\]
Then $f_\rho\in\ctemp_\varepsilon(V^2)$. In particular it is
$\SO(V,q)$-invariant, so the cited result makes it
$\OO(V,q)$-invariant. Since it also transforms by
$\varepsilon$, we have $f_\rho=0$. Finite sums of product
Schwartz densities are dense in the Schwartz densities on
$\R^2\times V^2$; thus $f=0$, proving
\begin{equation}\label{eq:real-matrix-odd}
                         \ctemp_\varepsilon(X)=0.
\end{equation}

The passage from $X$ to $X_E^\times$ follows from
\cite[Lemma~2.8]{JSZ}, restated in~\eqref{eq:real-localization}:
if a Nash group $L$ acts on
a Nash manifold $M$, $p:M\to N$ is an $L$-invariant Nash map,
and $\ctemp_\theta(M)=0$, then
$\ctemp_\theta(p^{-1}(N_0))=0$ for every Nash open
$N_0\subset N$. Here take $L=\OO(V,q)$, $\theta=\varepsilon$,
$M=X$, $N=\C$, $N_0=\C^\times$, and
\[
 p(A,B)=\det_{\C}(A+iB)
       =(a+ib)^2-\tfrac12\tr\bigl((A_0+iB_0)^2\bigr).
\]
The last expression is $L$-invariant because $L$ preserves
$q$ and its associated bilinear form. The map $p$ is polynomial
in real coordinates, and $\C^\times$ is Nash open in $\C$.
Consequently~\eqref{eq:real-matrix-odd} gives
\[
                   \ctemp_\varepsilon(X_E^\times)=0.
\]

Therefore $T_0-\tau_0^*T_0$ transforms by $\varepsilon$ and
must vanish. This proves transpose-invariance of $T$ for $F=\R$.

For non-Archimedean $F$, the required invariance on $Y$ follows
from~\cite[Proposition~4.5]{Prasad}: every distribution on $Y$
invariant under simultaneous conjugation is also invariant under
$(A,B)\mapsto(\operatorname{adj}(A),\operatorname{adj}(B))$.
Then, the identity $\operatorname{adj}(A)=J_0A^tJ_0^{-1}$ and
conjugation-invariance give simultaneous transpose-invariance of a $\mathrm{GL}_2(F)$-adjoint invariant distribution on $Y$.

The invariance on $Y$ extends to $X$ by scalar translation.
Indeed, consider the invariant open subsets
\[
 V_{r,s}=\{(A,B)\in X:\det(A+rI_2)\det(B+sI_2)\ne0\},
 \qquad r,s\in F.
\]
The map $(A,B)\mapsto(A+rI_2,B+sI_2)$ identifies $V_{r,s}$
with $Y$ and commutes with conjugation and transpose.
These subsets cover $X$, since $F$ is infinite and each
of the two determinant polynomials has only finitely many roots.
An invariant distribution negated by transpose consequently
vanishes on every $V_{r,s}$, and hence on $X$.

For the passage to $X_E^\times$, each fiber of
$(A,B)\mapsto\det(A+\delta B)$ over $E^\times$
is closed in $X$ and is preserved by conjugation and transpose.
By~\eqref{eq:closed-support-tools}, pushforward under this closed
inclusion embeds its invariant odd distributions into
the corresponding space on $X$, which is zero. Applying
Theorem~\ref{thm:localization}, the localization principle
of~\cite{Bernstein}, on $X_E^\times$ therefore proves the required
vanishing. No assumption on the residue characteristic is used.
\end{proof}

\begin{lemma}\label{lem:anisotropic-linear}
Suppose $H=\diag(1,d)$ is anisotropic. In the non-Archimedean
case, let $T$ be a distribution on $\GL_2(E)$; in the real case,
let $T$ be a generalized function on $\GL_2(\C)$. If $T$ is
invariant under conjugation by $K$, then it is invariant under
$\tau_2(z)=H^{-1}z^tH$. No temperedness hypothesis is needed
in the real case.
\end{lemma}
\begin{proof}
The quotient $K/E^\times$ is compact, which permits averaging
test functions over the conjugation action. Indeed, it identifies
with the closed subset
\[
 \{[k]\in\mathbb P(M_2(E)):k^\dagger Hk\in FH\}
\]
of the compact projective space $\mathbb P(M_2(E))\simeq\mathbb P^3(E)$.
If a nonzero matrix $k$ satisfies $k^\dagger Hk=\lambda H$,
then $\lambda=0$ would make both columns of $k$ isotropic for $H$,
forcing $k=0$. Thus $\lambda\ne0$, and the same identity makes
$k$ invertible. The displayed subset is therefore exactly
$K/E^\times$.

To see that $z$ and $\tau_2(z)$ lie in the same $K$-orbit,
consider
$V_0=\{A\in M_2(E):A^\dagger H=HA,\ \tr A=0\}$.
This is a three-dimensional $F$-space, whose elements and
quadratic form are
\[
 A=\begin{pmatrix}r&d\bar u\\u&-r\end{pmatrix},\quad
 r\in F,\ u\in E,\qquad q(A)=\tr(A^2)=2(r^2+dN(u)).
\]
The form $q$ is anisotropic because $H$ is anisotropic.
Choose $\delta\in E^\times$ with
$\bar\delta=-\delta$. Put $c=\tfrac12\tr z$ and $z_0=z-cI_2$.
Then
\[
 z=cI_2+A_0+\delta B_0,\qquad
 A_0=\frac{z_0+H^{-1}z_0^\dagger H}{2},\quad
 B_0=\frac{z_0-H^{-1}z_0^\dagger H}{2\delta}
\]
with $A_0,B_0\in V_0$.

A nonzero vector perpendicular to both $A_0$ and $B_0$ gives a
conjugation sending them to their negatives. The trace form is
$F$-valued on $V_0$, so
$\tr(vA_0)=\tr(vB_0)=0$ are two $F$-linear conditions on $v$.
Since $\dim_F V_0=3$, they admit a nonzero solution $v\in V_0$.
Cayley--Hamilton and polarization give
\[
 v^2=\tfrac12q(v)I_2,\qquad
 vX+Xv=\tr(vX)I_2\quad(X\in V_0).
\]
Anisotropy gives $q(v)\ne0$, so $v$ is invertible and
$v^\dagger Hv=Hv^2=\tfrac12q(v)H$; hence $v\in K$.
The second identity now gives
$vA_0v^{-1}=-A_0$ and $vB_0v^{-1}=-B_0$.

The desired conjugating matrix is obtained using
\[
 R=\begin{pmatrix}0&-d\\1&0\end{pmatrix},\qquad
 R^\dagger HR=dH,\qquad
 RXR^{-1}=-\tau_2(X)\quad(X\in V_0),
\]
where the last two identities follow by direct multiplication.
Consequently $k=Rv\in K$ satisfies
\begin{equation}\label{eq:anisotropic-conjugacy}
 kzk^{-1}=cI_2+\tau_2(A_0)+\delta\tau_2(B_0)=\tau_2(z).
\end{equation}

The conjugacy in~\eqref{eq:anisotropic-conjugacy} gives
invariance of $T$ by averaging, as follows. Over $\R$, fix an
additive density $dZ$ on $M_2(\C)$.
Conjugation has complex determinant $1$ and $\tau_2$ has complex
determinant $-1$, so both preserve $dZ$; in this case write
$T(\varphi)$ for $T(\varphi\,dZ)$.
Give $K/E^\times$ Haar measure $dk$ of total mass one and set
\[
 \mathcal A\varphi(z)=\int_{K/E^\times}\varphi(kzk^{-1})\,dk.
\]
For a compactly supported test function $\varphi$, this average
is again smooth, or locally constant in the non-Archimedean
case. Its support lies in the compact set
$(K/E^\times)\cdot\supp\varphi$. Moreover, $\mathcal A\varphi$
is constant on $K$-conjugacy orbits, so
\eqref{eq:anisotropic-conjugacy} gives
$\mathcal A\varphi(\tau_2(z))=\mathcal A\varphi(z)$.

Averaging also commutes with $\tau_2$. Indeed,
$[k]\mapsto[\tau_2(k)^{-1}]$ is an automorphism of $K/E^\times$
and preserves its normalized Haar measure. Using
\[
 \tau_2(kzk^{-1})
   =\tau_2(k)^{-1}\tau_2(z)\tau_2(k)
\]
in the defining integral gives
$\mathcal A(\varphi\circ\tau_2)
 =(\mathcal A\varphi)\circ\tau_2=\mathcal A\varphi$.
Finally, conjugation-invariance and continuity of $T$ give
\begin{align*}
 T(\mathcal A\varphi)
 &=\int_{K/E^\times}T\bigl(\varphi(k\,\cdot\,k^{-1})\bigr)\,dk
   =T(\varphi),\\
 T(\varphi\circ\tau_2)
 &=T\bigl(\mathcal A(\varphi\circ\tau_2)\bigr)
   =T(\mathcal A\varphi)=T(\varphi).
\end{align*}
This proves the lemma in both settings.
\end{proof}

\begin{proposition}\label{prop:open}
On $O$ there are no nonzero $\widetilde\chi$-equivariant tempered
generalized functions in the real case, or distributions in the
non-Archimedean case.
\end{proposition}
\begin{proof}
Restriction to the matrices $L_z$ reduces the proposition to
Lemmas~\ref{lem:linear} and~\ref{lem:anisotropic-linear}.
To use generalized functions in both settings, let $T$ be a
$\widetilde\chi$-equivariant distribution on $O$ in the
non-Archimedean case. Choose a bi-invariant Haar density $dg$
on $G$ and associate to $T$ the generalized function
$f(\varphi\,dg)=T(\varphi)$. The density is preserved by the
left and right translations defining the $\HH$-action.
It is also preserved by $\tau$: its image is a positive multiple
of $dg$, and $\tau^2=1$ forces this multiple to be one.
Thus $f$ has the same $\widetilde\chi$-equivariance as $T$.
In the real case, let $f$ be a $\widetilde\chi$-equivariant
tempered generalized function on $O$.

For this restriction, consider the action map
\[
 \Sigma=\{L_z:z\in\GL_2(E)\},\qquad
 \rho:\widetilde\HH\times\Sigma\longrightarrow O,
       \quad(h,L_z)\longmapsto h\cdot L_z.
\]
The Gauss decomposition $O=UM\tau(U)$ and the identity
$m(a,k)=L_{ak^{-1}}D(k)$ show that
\[
 (U\times S)\times\Sigma\longrightarrow O,
 \qquad ((u,s),L_z)\longmapsto uL_z\tau(s)
\]
is an analytic isomorphism, and a Nash isomorphism over $\R$.
Indeed, writing $s=vD(\tau_2(k))$ gives the Gauss coordinates
$uL_zD(k)\tau(v)=u m(zk,k)\tau(v)$, so $z=ak^{-1}$.
Consequently $\rho$ is surjective and its differential is
surjective at $(1,L_z)$; equivariance gives the same at every
point. Thus $\Sigma$ is a $\widetilde\HH$-slice.

The slice is preserved by $\tau$ and by the subgroup
$\iota(K)\subset\HH$ defined, for $k\in K$, by
\[
 \iota(k)=\bigl(D(k),D(\tau_2(k^{-1}))\bigr)\in\HH.
\]
The subgroup $\iota(K)$ is closed, and Nash over $\R$.
Direct multiplication gives
\[
 \iota(k)\cdot L_z=D(k)L_zD(k^{-1})=L_{kzk^{-1}},\qquad
 \tau(L_z)=L_{\tau_2(z)}.
\]
Moreover, since $\tau_2$ preserves determinant and similitude
factor,
\[
 \widetilde\chi(\iota(k))
   =\omega(k)\omega(\tau_2(k^{-1}))=1,\qquad
 \widetilde\chi(\tau)=-1.
\]

Since $\rho$ is a surjective submersion, restriction to $\Sigma$
is injective. More precisely, the result in
\cite[Lemma~2.13]{JLX} in the non-Archimedean case, and
\cite[Lemma~2.7]{JSZ} in the real case, defines restriction by
\[
                 \rho^*f=\widetilde\chi\otimes t,
\]
where $t$ is a generalized function on $\Sigma$. This restriction
preserves the restricted character under any closed subgroup
preserving $\Sigma$. In the Nash setting it
also preserves temperedness, by the paragraph following
\cite[Lemma~2.8]{JSZ}; see
Section~\ref{subsubsec:real-localization}.
Applying these statements to $\rho$, $\iota(K)$, and
$\langle\tau\rangle$, and identifying $\Sigma$ with $\GL_2(E)$,
we obtain
\[
             t(kzk^{-1})=t(z),\qquad t(\tau_2(z))=-t(z)
\]
in the sense of generalized functions.

The two matrix lemmas force $t$ to vanish. In the
non-Archimedean case, a bi-invariant Haar density on $\GL_2(E)$
identifies $t$ with a distribution; this density is preserved
by conjugation and $\tau_2$, so both displayed properties are
unchanged. Lemma~\ref{lem:linear} for $H=J$, or
Lemma~\ref{lem:anisotropic-linear} for anisotropic $H$, states
that such a conjugation-invariant distribution, or tempered
generalized function over $\R$, is $\tau_2$-invariant.
Thus $t$ is both fixed and negated by $\tau_2$, whence $t=0$.
Injectivity of restriction gives $f=0$, proving the proposition.
\end{proof}

\subsection{The quasi-split stratum 3221}
\label{sec:common3221}

Throughout this subsection $H=J$. The boundary argument on
$T=G_{3221}$ uses a closed filtration
$T\supset C\supset C_0\supset Z$.
On $T\setminus C$ and $C_0\setminus Z$, the stabilizers contain
unipotent elements on which $\chi$ is nontrivial. The remaining
pieces $C^\times=C\setminus C_0$ and $Z$ admit normal forms
used in the distribution arguments in Section~\ref{sec:padic} and
the normal-vector calculations in Section~\ref{sec:real}.
All formulas hold in both local-field settings. Recall that
$E^-=\{z\in E:\bar z=-z\}$ and $N(z)=z\bar z$.

Let $w=(1\ 4)(3\ 6)$ be the permutation matrix representing $T$,
and put $L_A=m(A,I_2)$. We write
\begin{equation}\label{eq:AB}
 g(A,B)=L_AwL_B,\qquad
 A=\begin{pmatrix}a&b\\c&d\end{pmatrix},\quad
 B=\begin{pmatrix}m&n\\p&q\end{pmatrix}\in\GL_2(E).
\end{equation}
By~\eqref{eq:rank-cover}, the map
\begin{equation}\label{eq:cover}
 \HH\times\GL_2(E)^2\longrightarrow T,\qquad
 (h,A,B)\longmapsto h\cdot g(A,B)
\end{equation}
is a surjective analytic submersion, and a Nash submersion
over $\R$.

\subsubsection{The filtration and stabilizer characters}

\begin{lemma}\label{lem:3221-filtration}
The conditions
\begin{align}
 C &: c=n, & C_0 &: c=n=0,\label{eq:C}\\
 Z &: c=n=0,\quad dq/a\in F^\times,\quad dq/m\in F^\times
                                                        \label{eq:Z}
\end{align}
on the inverse images under~\eqref{eq:cover} define a
$\widetilde\HH$-stable closed filtration
\begin{equation}\label{eq:filtration}
                 T\supset C\supset C_0\supset Z\supset\varnothing.
\end{equation}
The submanifolds $T,C,C_0,Z$ are smooth of dimensions
$35,33,31,29$ over $F$, respectively.

The set $C$ is the locus where $\chi$ is trivial on the
$U\times U$-stabilizer; $C_0$ is the locus in $C$ where $\xi$
is trivial on its first projection; and $Z$ is the locus in
$C_0$ where $\chi$ is trivial on the full $\HH$-stabilizer.
On $T\setminus C$ and $C_0\setminus Z$, the character is
nontrivial on some unipotent element of the stabilizer.
\end{lemma}
\begin{proof}
The loci $C$ and $C_0$ are determined by the unipotent stabilizers
and the restrictions of $\chi$ and $\xi$. Let $e_{ij}$ denote
the $2\times2$ matrix units, and put $u_z=u(ze_{12},0)$.
At $w$, the condition
$wu(X,Y)w\in U^t$ forces $X_{11}=X_{21}=X_{22}=0$ and $Y=0$.
The remaining parameter is $X_{12}=z$, whose quadratic term
vanishes. Thus the stabilizer is described by
$u_zw=wu_z^t$, $z\in E$. Conjugating by the Levi factors gives
the stabilizing pair at $g(A,B)$:
\[
 \left(L_Au_zL_A^{-1},(L_B^tu_zL_B^{-t})^{-1}\right).
\]
The two contributions to its unipotent character are $cz$ and
$-nz$, so its character is
\begin{equation}\label{eq:firstchar}
                     \psi\bigl(\Tr((c-n)z)\bigr).
\end{equation}

The intrinsic descriptions of $C$ and $C_0$ follow from
nondegeneracy of the trace pairing. Indeed,~\eqref{eq:firstchar}
is trivial for every $z$ exactly when $c=n$. On this locus,
$\xi$ on the first projection of the stabilizer is
$\psi(\Tr(cz))$, which is trivial for every $z$ exactly when
$c=0$. Since $U\times U$ is normal in $\HH$ and $\eta$ is a
character of $S$, both descriptions are preserved by $\HH$.

To identify $Z$, we compute the full $\HH$-stabilizer, beginning
with $P\cap wP^tw$. Its elements are precisely
\begin{equation}\label{eq:3221-intersection}
\begin{gathered}
h=m(H_0,k)u(X,Y),\qquad
H_0=\begin{pmatrix}\alpha&\beta\\0&\epsilon\end{pmatrix},
\qquad
k=\begin{pmatrix}k_{11}&k_{12}\\0&k_{22}\end{pmatrix},\\
X=\begin{pmatrix}0&z\\0&v_0\end{pmatrix},\qquad Y=re_{12},\\
\alpha,\epsilon,k_{11},k_{22}\in E^\times,\qquad
\overline{k_{11}}k_{22}=\mu=\lambda(k)\in F^\times,\qquad
k_{12}/k_{22},\,r\in E^-.
\end{gathered}
\end{equation}
Indeed, $whw\in P^t$ forces $H_0$ and $k$ to be upper
triangular, the first column of $X$ to vanish, and $Y=re_{12}$.
The unitary equation for $k$ gives the displayed conditions on
its entries. Since $XX'=0$, the remaining equation is
$r+\bar r=0$. Conversely, these conditions imply all the
required matrix entries and unitary equations.

The conditions defining $S$ translate into two equations
between diagonal blocks. Put $\gamma=\alpha z+\beta v_0$ and
$\nu=\mu/\bar\alpha$. The first two diagonal blocks of $whw$
are
\[
 C_1=\begin{pmatrix}k_{22}&0\\\epsilon v_0&\epsilon\end{pmatrix},
 \qquad
 C_2=\begin{pmatrix}\nu&0\\\alpha r&\alpha\end{pmatrix}=l^t.
\]
For $s=L_AhL_A^{-1}$ and $v=(L_B^{-1}whwL_B)^t$, the
conditions $s,v\in S$ are exactly
\begin{equation}\label{eq:3221-stabilizer-equations}
                   kA=AH_0,\qquad C_1B=Bl^t.
\end{equation}
The third diagonal blocks are then forced by the unitary
equation. These elements satisfy $sg(A,B)=g(A,B)v^t$, so
$(s,v^{-1})$ is a stabilizing pair. Conversely, every such
pair arises from this parabolic intersection and satisfies
the two equations.

For this stabilizing pair, the left unipotent trace is
$cz+dv_0$. The lower middle-left block of $whw$ is
$\left(\begin{smallmatrix}0&0\\\gamma&\beta\end{smallmatrix}\right)$,
giving right trace $(n\gamma+q\beta)/\alpha$. Including the
two Levi factors gives
\begin{equation}\label{eq:3221-full-character}
\chi(s,v^{-1})=
\frac{\omega(k)}{\omega(l)}
\psi\!\left(\Tr
\left((c-n)z+dv_0-\frac{\beta(nv_0+q)}{\alpha}\right)\right).
\end{equation}

On $C_0$, the stabilizer equations determine all the parameters
except a central scalar and three unipotent parameters.
Here $a,d,m,q\ne0$. The diagonal equations force
$\alpha=\epsilon=k_{11}=k_{22}$, and the remaining equations give
\[
 u_0=k_{12}/k_{22}\in E^-,\qquad
 \beta=k_{22}u_0d/a,\qquad v_0=qr/m.
\]
Conversely, these relations solve~\eqref{eq:3221-stabilizer-equations}.
In particular, taking $k_{22}=1$ gives the unipotent stabilizer
family
\begin{align*}
 k&=\begin{pmatrix}1&u_0\\0&1\end{pmatrix},&
 H_0&=\begin{pmatrix}1&u_0d/a\\0&1\end{pmatrix},\\
 X&=\begin{pmatrix}0&z\\0&qr/m\end{pmatrix},&
 Y&=re_{12},\qquad u_0,r\in E^-,\quad z\in E,
\end{align*}
with $h,s,v$ defined above. Its character is
\begin{equation}\label{eq:secondchar}
 \chi(s,v^{-1})=
 \psi\left(\Tr\left(dq\left(\frac r m-\frac{u_0}a\right)\right)\right).
\end{equation}

For general $k_{22}$, the matrices $h,s,v$ are scalar multiples
by $k_{22}$ of this family, and
$l=k_{22}\left(\begin{smallmatrix}1&r\\0&1\end{smallmatrix}\right)$.
The scalar factors cancel in the character of $(s,v^{-1})$,
so~\eqref{eq:secondchar} gives the full stabilizer character.
The annihilator of $E^-$ under $(x,y)\mapsto\Tr(xy)$ is $F$.
Consequently the character is trivial exactly when
$dq/a,dq/m\in F^\times$, proving the intrinsic description and
$\HH$-stability of $Z$. Off $Z$, the family with $k_{22}=1$
already supplies a unipotent stabilizer with nontrivial
character, just as~\eqref{eq:firstchar} does off $C$.

Closedness and smoothness follow from the defining equations
in the parameter space. The conditions $c-n=0$ and then $c=0$
are independent $E$-equations. On $C_0$, the conditions $dq/a,dq/m\in F$
each impose one $F$-equation; varying $a$ and $m$ independently
shows that their differentials are independent.
Since $a,d,m,q$ are nonzero, the conditions are closed in
their respective predecessors. Their inverse images are
saturated for~\eqref{eq:cover}, so local product coordinates
for this submersion give closed embedded smooth submanifolds
of $T$, of the stated dimensions.

Finally, the intrinsic descriptions also give transpose
stability. Transpose interchanges the two stabilizer
projections and preserves $\chi$. On $C$, triviality of
$\xi$ on one projection is equivalent to triviality on the
other. Thus $C$, $C_0$, and $Z$ are $\widetilde\HH$-stable.
\end{proof}

\subsubsection{Normalization of $C\setminus C_0$}

The normal form below supplies a slice for the non-Archimedean
argument and the tangent directions needed over $\R$.

\begin{lemma}\label{lem:3221-normal-form}
For $(\zeta,x,y)\in E^\times\times F^2$, put
\begin{equation}\label{eq:normalized}
 g_\zeta(x,y)=L_{A_x}wL_{B_{y,\zeta}},\qquad
 A_x=\begin{pmatrix}x/2&1\\1&0\end{pmatrix},\quad
 B_{y,\zeta}=\begin{pmatrix}y/2&1\\\zeta&0\end{pmatrix}.
\end{equation}
Every $g(A,B)$ with $c=n\ne0$ is $\HH$-equivalent to
$g_\zeta(x,y)$ with
\[
 x=\Tr(a/c),\qquad y=N(c)\Tr(m/c),\qquad
 \zeta=\det A\det B\,\bar c/c^2.
\]
Moreover, the action map
\[
 \HH\times(E^\times\times F^2)\longrightarrow C^\times,\qquad
 (h,\zeta,x,y)\longmapsto h\cdot g_\zeta(x,y)
\]
is a surjective analytic submersion, and a Nash submersion
over $\R$.
\end{lemma}
\begin{proof}
The matching equations from the stabilizer calculation also
give the required orbit equivalence. Set
\[
 \theta=x/2-a/c,\qquad r=m/c-y/(2N(c)),\qquad
 k=\begin{pmatrix}1&\theta\\0&1\end{pmatrix},\qquad
 H_0=\begin{pmatrix}c&d\\0&-\det A/c\end{pmatrix},
\]
where $\theta,r\in E^-$, and take
\[
 h=m(H_0,k)u\left(\begin{pmatrix}0&0\\0&-q/c\end{pmatrix},re_{12}\right).
\]
The diagonal blocks $C_1,C_2$ of $whw$ satisfy
$kA=A_xH_0$ and $C_1B=B_{y,\zeta}C_2$.
As in~\eqref{eq:3221-stabilizer-equations}, these identities show
that $s=L_{A_x}hL_A^{-1}$ and
$v=(L_{B_{y,\zeta}}^{-1}whwL_B)^t$ belong to $S$ and satisfy
$sg(A,B)=g_\zeta(x,y)v^t$.

Multiplication of the three factors in~\eqref{eq:normalized}
gives the matrix needed for the later differential calculations:
\begin{equation}\label{eq:normalizedmatrix}
 g_\zeta(x,y)=
 \begin{pmatrix}
 \zeta&0&0&x/2&0&0\\
 0&0&0&1&0&0\\
 0&0&0&0&1&0\\
 y/2&1&0&0&0&0\\
 0&0&1&0&xy/(4\bar\zeta)&-x/(2\bar\zeta)\\
 0&0&0&0&-y/(2\bar\zeta)&1/\bar\zeta
 \end{pmatrix}.
\end{equation}

Submersivity follows from the smooth dependence of these
normalizing elements on $A$ and $B$. Their formulas are smooth
on $c\ne0$, so the restriction of~\eqref{eq:cover} factors
smoothly through the displayed action map. On the normalized family itself, $s=v=1$.
Differentiating this factorization at those points proves
surjectivity of the new differential at every
$(1,\zeta,x,y)$, and $\HH$-equivariance gives it everywhere.
The formulas are Nash over $\R$. In particular, the orbit
directions and derivatives of~\eqref{eq:normalizedmatrix}
span the tangent space to $C^\times$, including where $x=0$
or $y=0$.
\end{proof}

\subsubsection{Diagonal normalization of $Z$}

On the last closed locus, the triangular matrices $A$ and $B$
can be made diagonal.

\begin{lemma}\label{lem:3221-diagonal-form}
Every $\HH$-orbit in $Z$ meets the family
\begin{equation}\label{eq:diagonalAB}
 A=\diag(a,d),\qquad B=\diag(m,q),\qquad
 dq/a,dq/m\in F^\times,\qquad a,d,m,q\in E^\times.
\end{equation}
The map $(h,a,d,m,q)\mapsto h\cdot g(\diag(a,d),\diag(m,q))$,
on $\HH$ times this parameter space, is a surjective analytic
submersion, and a Nash submersion over $\R$.
\end{lemma}
\begin{proof}
Since $Z\subset C_0$, the matrix $A$ is upper triangular and
$B$ is lower triangular. Their off-diagonal entries can be
removed without changing $a,d,m,q$: put
$A'=\diag(a,d)$, $B'=\diag(m,q)$, and
\[
 h=m\left(\begin{pmatrix}1&b/a\\0&1\end{pmatrix},I_2\right)
        u\left(\begin{pmatrix}0&0\\0&-p/m\end{pmatrix},0\right).
\]
The matching identities are
\[
 A=A'\begin{pmatrix}1&b/a\\0&1\end{pmatrix},\qquad
 \begin{pmatrix}1&0\\-p/m&1\end{pmatrix}B=B'.
\]
They imply that $s=L_{A'}hL_A^{-1}$ and
$v=(L_{B'}^{-1}whwL_B)^t$ belong to $S$ and satisfy
$sg(A,B)=g(A',B')v^t$, proving the orbit equivalence.

These normalizing elements depend smoothly on the parameters,
preserve the conditions defining $Z$, and are the identity
when $b=p=0$.
The restricted map~\eqref{eq:cover} therefore factors through
the action map in the statement, and the differentiation
argument of Lemma~\ref{lem:3221-normal-form} applies.
The formulas are again Nash over $\R$. Thus the orbit and
parameter directions span the tangent spaces to $Z$ used
in Lemma~\ref{lem:Zmetric}.
\end{proof}
\section{Multiplicity one over a non-Archimedean local field}
\label{sec:padic}

Throughout this section $F$ is non-Archimedean. Our goal is to prove the
non-Archimedean assertion of Theorem~\ref{thm:distribution}, namely
\[
  \mathcal D'(G)^{\widetilde\HH,\widetilde\chi}=0.
\]
Proposition~\ref{prop:open} gives this vanishing on the open cell. The
remaining task is to exclude distributions supported on the boundary
strata of Lemma~\ref{lem:rank-common}. Two arguments recur. The
character obstruction in Lemma~\ref{lem:padic-orbits} says that an
orbit carries no nonzero equivariant distribution if its stabilizer
contains a unipotent element on which the character is nontrivial, or
an element on which the extended character is $-1$. The second
situation is produced by~\eqref{eq:padic-involutive-transporter}: an
element of $\HH$ carrying $x$ to $\tau(x)$ with character one gives an
element of the extended stabilizer with character $-1$.

The closed stratum $G_{2000}$ is treated uniformly for the two
Hermitian forms. For the quasi-split group we then prove vanishing on
$G_{2110}$, $G_{3221}$, and $G_{3111}$. The middle stratum contains
one piece for which the stabilizer character is trivial; the normal
form of Section~\ref{sec:common3221} reduces that piece to a
distribution problem on $F^2$. In the non-quasi-split case only
$G_{3111}$ remains after the open and closed strata. Once every
stratum has been treated, the support statement
\eqref{eq:closed-support-tools} passes the vanishing through the finite
closed filtration.

\subsection{Vanishing on the closed stratum $G_{2000}$}

\begin{proposition}\label{prop:padic-2000}
For either Hermitian form, there are no nonzero odd
$(\HH,\chi)$-equivariant distributions on $G_{2000}$. Equivalently,
\[
 \mathcal D'(G_{2000})^{\widetilde\HH,\widetilde\chi}=0.
\]
\end{proposition}
\begin{proof}
Lemma~\ref{lem:padic-orbits} will be applied orbit by orbit. Recall
that it gives vanishing when the stabilizer contains a unipotent
element with nontrivial character, or an element on which the extended
character is $-1$.

Let
\[
 J_3=w_3\otimes I_2.
\]
Every $\HH$-orbit in $G_{2000}$ has a representative
\begin{equation}\label{eq:ns-padic-2000-representative}
 x_z=L_zJ_3=
 \begin{pmatrix}0&0&z\\0&I_2&0\\z^*&0&0\end{pmatrix},
                         \qquad z\in\GL_2(E).
\end{equation}
For $X\in M_2(E)$ put $X'=-H^{-1}X^\dagger H$ and choose
$Y=\frac12XX'$.  Then $u=u(X,Y)\in U$.  The element
\[
                   v=\tau(x_z^{-1}ux_z)
\]
also belongs to $U$, and $(u,v^{-1})$ fixes $x_z$.  Its character is
\begin{equation}\label{eq:ns-padic-2000-character}
 \chi(u,v^{-1})
 =\psi\!\left(\operatorname{Tr}_{E/F}\tr\bigl((I_2+z^{-1})X\bigr)\right).
\end{equation}
For completeness, block multiplication gives
$(x_z^{-1}ux_z)_{21}=X'z^*$ and
$(x_z^{-1}ux_z)_{32}=z^{-1}X$.
The $(1,2)$ block of $v$ is
$X_v=\tau_2(X'z^*)$. Consequently
\[
 \operatorname{Tr}_{E/F}\tr(X_v)
 =-\operatorname{Tr}_{E/F}\tr(z^{-1}X).
\]
For weighted transpose use $z^*=H^{-1}(z^\dagger)^{-1}H$;
for ordinary transpose use $H=J=J^{-1}$. In both cases
the field trace is unchanged by conjugation and transpose.  This proves
\eqref{eq:ns-padic-2000-character}.

If $z\neq-I_2$, the nondegeneracy of
$(A,X)\mapsto\operatorname{Tr}_{E/F}\tr(AX)$ supplies an $X$ for which
\eqref{eq:ns-padic-2000-character} is different from $1$. Hence the
stabilizer contains a unipotent element on which $\chi$ is nontrivial.
The first assertion of Lemma~\ref{lem:padic-orbits} therefore gives
vanishing on this orbit. If $z=-I_2$, then
\[
                  x_{-I_2}=\tau(x_{-I_2}).
\]
Thus $\tau$ fixes $x_{-I_2}$ and
$\widetilde\chi(\tau)=-1$. This is the second stabilizer condition in
Lemma~\ref{lem:padic-orbits}, so the orbit again admits no nonzero odd
equivariant distribution. Since every orbit in $G_{2000}$ is covered
by one of these two cases, the final assertion of that lemma gives the
proposition.
\end{proof}

\subsection{The quasi-split group}\label{sec:qs-padic}

Here $H=J$, $W=w_6$, and $\tau$ is ordinary transpose. By
Propositions~\ref{prop:open} and~\ref{prop:padic-2000}, it remains to
prove vanishing on $G_{2110}$, $G_{3221}$, and $G_{3111}$. We state
each vanishing result before its calculation, so the final support
argument can refer to them directly.

\subsubsection{Vanishing on $G_{2110}$}

\begin{proposition}\label{prop:padic-2110}
Assume that $H$ is quasi-split. Then there are no nonzero odd
$(\HH,\chi)$-equivariant distributions on $G_{2110}$. Equivalently,
\[
 \mathcal D'(G_{2110})^{\widetilde\HH,\widetilde\chi}=0.
\]
\end{proposition}
\begin{proof}
The proof first identifies the points for which the unipotent
stabilizer character is nontrivial. The remaining closed subset is
handled by transpose symmetry.

Let $w=(1\ 3)(2\ 5)(4\ 6)$ be the representative of this stratum. By \eqref{eq:rank-cover}, every point of $G_{2110}$ lies in the $\HH$-orbit of
\[
x(A,B)=L_AwL_B,\qquad A=\begin{pmatrix}a&b\\c&d\end{pmatrix},\quad B=\begin{pmatrix}m&n\\p&q\end{pmatrix}\in\GL_2(E).
\]
Put $N_U:=U\times U$. The $N_U$-stabilizer determines where the character obstruction applies.

Since $L_A$ and $L_B^t$ normalize $U$, it is enough to solve
\[
u_0w=wv_0^t,\qquad u_0,v_0\in U.
\]
Write $v_0=u(X,Y)$ with $X=(x_{ij})$ and $Y=(y_{ij})$. As $w=w^t=w^{-1}$, the condition is equivalent to $wv_0^tw\in U$. Comparing the entries which must vanish in an upper block-unipotent matrix gives $x_{12}=x_{22}=y_{12}=0$. Thus
\[
X=\begin{pmatrix}t&0\\s&0\end{pmatrix},\qquad Y=\begin{pmatrix}r&0\\\epsilon&y\end{pmatrix}.
\]
For this $X$, the associated block $X'=-JX^\dagger J$ satisfies $(X')^\dagger JX'=0$. The remaining unitary equation therefore reduces to $JY+Y^\dagger J=0$, which gives $y=-\bar r$ and $\epsilon+\bar\epsilon=0$. Consequently all solutions are precisely
\[
\begin{aligned}
v_0&=u\left(\begin{pmatrix}t&0\\s&0\end{pmatrix},\begin{pmatrix}r&0\\\epsilon&-\bar r\end{pmatrix}\right),\\
u_0&=u\left(\begin{pmatrix}t&0\\r&0\end{pmatrix},\begin{pmatrix}s&0\\\epsilon&-\bar s\end{pmatrix}\right),
\end{aligned}
\]
where $t,r,s\in E$ and $\epsilon\in E^-$. Conversely, direct substitution verifies $u_0w=wv_0^t$ for every such choice of parameters. This proves that the list is exhaustive. In particular, the unipotent stabilizer at $w$ is the vector group $(E,+)^3\times(E^-,+)$, of dimension seven over $F$.

For $x=x(A,B)$, every solution of $ux=xv^t$ is obtained uniquely by setting
\[
u=L_Au_0L_A^{-1},\qquad v=L_B^tv_0L_B^{-t}.
\]
The corresponding stabilizing element of $N_U$ is $(u,v^{-1})$. The first adjacent blocks of $u$ and $v$ are respectively
\[
A\begin{pmatrix}t&0\\r&0\end{pmatrix},\qquad B^t\begin{pmatrix}t&0\\s&0\end{pmatrix}.
\]
Their traces are $at+br$ and $mt+ps$. Hence
\begin{equation}\label{eq:padic-2110char}
\chi(u,v^{-1})=\psi\bigl(\Tr((a-m)t+br-ps)\bigr).
\end{equation}
The three parameters $t,r,s$ vary independently over $E$. By nondegeneracy of the trace pairing, the character is trivial on the entire unipotent stabilizer if and only if $a=m$ and $b=p=0$.

Define intrinsically
\[
C_N=\{x\in G_{2110}:\chi|_{(N_U)_x}=1\}.
\]
Its inverse image under the covering \eqref{eq:rank-cover} is given by $a=m$ and $b=p=0$.

The subset $C_N$ is closed in $G_{2110}$. Let $\mathfrak n=\mathfrak u\oplus\mathfrak u$. The infinitesimal stabilizer at $x$ is
\[
\mathfrak n_x=\{(Z_1,Z_2)\in\mathfrak n:Z_1x+xZ_2^t=0\}.
\]
The preceding calculation and Levi transport show that these spaces have constant dimension seven. They therefore form an algebraic vector bundle over $G_{2110}$.

Consider the $F$-linear additive functional
\[
\ell(Z_1,Z_2)=\Tr\left(\tr (Z_1)_{12}+\tr (Z_2)_{12}\right),
\]
where the subscripts denote the upper adjacent $2\times2$ blocks. Since the stabilizers are smooth vector groups, triviality of their character is equivalent to $\ell|_{\mathfrak n_x}=0$. Indeed, a nonzero $F$-linear functional has image $F$, on which $\psi$ is not identically one. Thus $C_N$ is the zero locus of an algebraic section of the dual stabilizer bundle, and hence is closed. It is stable under $\HH$ by normality of $N_U$ and invariance of its character. It is also stable under $\tau$, which exchanges the two factors.

On $G_{2110}\setminus C_N$, the stabilizer contains a unipotent
element on which $\chi$ is nontrivial. Lemma~\ref{lem:padic-orbits}
states that such an orbit carries no nonzero equivariant distribution;
its orbitwise-to-global assertion therefore gives
\[
\mathcal D'(G_{2110}\setminus C_N)^{\widetilde\HH,\widetilde\chi}=0.
\]

On $C_N$ the stabilizer character no longer gives vanishing, so we realize transpose by an element of character one. The defining parameter condition gives
\[
A=\begin{pmatrix}a&0\\c&d\end{pmatrix},\qquad B=\begin{pmatrix}a&n\\0&q\end{pmatrix},\qquad a,d,q\ne0.
\]
For $r,s\in E$ put
\[
Q(r,s)=u\left(\begin{pmatrix}0&r\\0&s\end{pmatrix},0\right).
\]
These matrices belong to $U$: their associated blocks satisfy $(X')^\dagger JX'=0$, so the unitary equation holds with $Y=0$.

Set
\[
Q_L=Q(0,d\bar n/\bar q),\qquad Q_R=Q(0,q\bar c/\bar d),\qquad z=d/\bar q.
\]
Direct block multiplication yields
\begin{equation}\label{eq:padic-2110normal}
Q_Lx(A,B)Q_R^t=y(a,z):=
\begin{pmatrix}
0&0&a&0&0&0\\
0&0&0&0&z&0\\
a&0&0&0&0&0\\
0&0&0&0&0&\bar a^{-1}\\
0&\bar z^{-1}&0&0&0&0\\
0&0&0&\bar a^{-1}&0&0
\end{pmatrix}.
\end{equation}
Here $a,z\in E^\times$. Thus every point of $C_N$ is $\HH$-conjugate, under the prescribed left-right action, to a matrix of this form. 

Put $k_z=\diag(1,N(z)^{-1})$. Since $N(z)\in F^\times$, we have $k_z^\dagger Jk_z=N(z)^{-1}J$, so $k_z\in K$. Define
\[
t_z=(D(k_z),D(k_z)^{-1})\in\HH.
\]
Using $z/N(z)=\bar z^{-1}$ and $N(z)/\bar z=z$, we obtain directly from \eqref{eq:padic-2110normal} that
\[
t_z\cdot y(a,z)=y(a,z)^t.
\]
Moreover,
\[
\chi(t_z)=\omega(k_z)\omega(k_z^{-1})=1,\qquad \sigma(t_z)=t_z^{-1}.
\]

The same property holds at every point of $C_N$. Let $x\in C_N$ and choose $h\in\HH$ such that $h\cdot x=y(a,z)$. Define $t_x=\sigma(h)^{-1}t_zh$. Since $\tau(h\cdot x)=\sigma(h)\cdot\tau(x)$ and $\chi\circ\sigma=\chi$, we obtain
\[
t_x\cdot x=\tau(x),\qquad \chi(t_x)=1,\qquad \sigma(t_x)=t_x^{-1}.
\]
Thus $t_x^{-1}\tau$ lies in the extended stabilizer of $x$ and has
$\widetilde\chi(t_x^{-1}\tau)=-1$. This is precisely the second
stabilizer condition in Lemma~\ref{lem:padic-orbits}; hence
\[
\mathcal D'(C_N)^{\widetilde\HH,\widetilde\chi}=0.
\]

Equation~\eqref{eq:closed-support-tools} identifies a distribution on
$G_{2110}$ supported on the closed subset $C_N$ with a distribution on
$C_N$. Since the odd distribution spaces vanish both on
$G_{2110}\setminus C_N$ and on $C_N$, the proposition follows.
\end{proof}

\subsubsection{Vanishing on $G_{3221}$}\label{subsec:plane-padic}

\begin{lemma}\label{lem:padic-plane}
Every distribution on $F^2$ invariant under
$(x,y)\mapsto(tx,t^{-1}y)$, $t\in F^\times$, is invariant under
$(x,y)\mapsto(y,x)$.
\end{lemma}
\begin{proof}
Let $T$ be invariant under the $F^\times$-action and odd under
$(x,y)\mapsto(y,x)$. Bernstein localization
(Theorem~\ref{thm:localization}) says that if the relevant equivariant
distribution space vanishes on every fiber of an invariant map, then
it vanishes on the whole space. Apply this to $p(x,y)=xy$. For
$c\ne0$, the fiber $p^{-1}(c)$ is the single $F^\times$-orbit
$(r,c/r)$, $r\in F^\times$. Its invariant distributions form the line
spanned by multiplicative Haar measure, and the involution
$r\mapsto c/r$ preserves that measure. Hence the odd part vanishes on
every nonzero fiber, and $T$ is supported on $Y=\{xy=0\}$.

The remaining assertion on $Y$ is
\cite[Lemma~4.6]{Prasad}: an $F^\times$-invariant distribution on the
union of the two coordinate axes is invariant under exchange of the
coordinates. For completeness, we recall the argument. On the two
punctured axes the restrictions of an odd invariant distribution have
the form $a\,d^\times x$ and $b\,d^\times y$, with $b=-a$. Let
$\mathfrak o$ be the ring of integers, $\varpi$ a uniformizer, and
$f=1_{Y\cap\mathfrak o^2}$. Invariance under
$(x,y)\mapsto(\varpi x,\varpi^{-1}y)$ gives $0=a-b$ after normalizing
the mass of $\mathfrak o^\times$ to one. Thus $a=b=0$. What remains is
supported at the origin, hence is a multiple of $\delta_0$, which is
even. Therefore $T=0$.
\end{proof}

\begin{proposition}\label{prop:padic-3221}
Assume that $H$ is quasi-split. Then there are no nonzero odd
$(\HH,\chi)$-equivariant distributions on $G_{3221}$. Equivalently,
\[
 \mathcal D'(G_{3221})^{\widetilde\HH,\widetilde\chi}=0.
\]
\end{proposition}
\begin{proof}
Take $w=(1\ 4)(3\ 6)$ and $g(A,B)=L_AwL_B$, where
$A=\left(\begin{smallmatrix}a&b\\c&d\end{smallmatrix}\right)$ and
$B=\left(\begin{smallmatrix}m&n\\p&q\end{smallmatrix}\right)$. Section~\ref{sec:common3221} gives the closed stable filtration
\[
 G_{3221}\supset C\supset C_0\supset Z,
\]
whose inverse images in the covering are
\[
 C:\ c=n,\qquad C_0:\ c=n=0,\qquad
 Z:\ c=n=0,\quad dq/a,dq/m\in F^\times.
\]

The pieces $G_{3221}\setminus C$ and $C_0\setminus Z$ are removed by the character obstruction.
Equation~\eqref{eq:firstchar} gives a unipotent stabilizer element with
nontrivial character at every point of $G_{3221}\setminus C$.
Equation~\eqref{eq:secondchar} gives the same conclusion on
$C_0\setminus Z$ after choosing the central scalar parameter to be one;
the annihilator of $E^-$ under the trace pairing is $F$.
Lemma~\ref{lem:padic-orbits} states that this nontrivial unipotent
stabilizer character forces vanishing, so both pieces contribute
nothing. The two remaining pieces are $C\setminus C_0$ and $Z$.

The subset $C\setminus C_0$ is the only part that requires a separate
distribution problem. Its normal form reduces it to the preceding
lemma.

For $c=n\ne0$, the normal form of Section~\ref{sec:common3221} gives
\[
 x=\Tr(a/c),\quad y=N(c)\Tr(m/c),\quad
 \zeta=\det(A)\det(B)\bar c/c^2,
\]
and the matrix $g_\zeta(x,y)$ in \eqref{eq:normalizedmatrix}, where
$(\zeta,x,y)\in E^\times\times F^2$. There is no nonvanishing
condition on $x$ or $y$. With the normalization convention
$sg(A,B)=g_\zeta(x,y)v^t$, its character is
\[
 \chi(s,v^{-1})=
 \chi_E(\bar c/c)\psi\bigl(\Tr(-dq/c)\bigr).
\]
Every pair $(s,v)\in\HH$ carrying one normalized point to
another has the unique form
\begin{align*}
 T(t,\lambda,z)&=h(t,\lambda)j_{x,y}(z),\\
 h(t,\lambda)&=
 \bigl(D(\diag(t\lambda,\lambda)),
       D(\diag((t\lambda)^{-1},\lambda^{-1}))\bigr),\\
 j_{x,y}(z)&=(U_x(z),U_y(-z)),\qquad
 U_x(z)=u\left(\begin{pmatrix}0&xz/2\\0&z\end{pmatrix},0\right),
\end{align*}
where $t\in F^\times$, $\lambda\in E^\times$, $z\in E$.
Their effect and character are
\begin{equation}\label{eq:padic-plane-transitions}
 (\zeta,x,y)\longmapsto(\zeta,tx,t^{-1}y),\qquad
 \chi(T(t,\lambda,z))=1.
\end{equation}
These are all the identifications among normalized points. To verify this, let an element of $P\cap wP^tw$ have
\[
 H_0=\begin{pmatrix}\alpha&\beta\\0&\epsilon\end{pmatrix},\quad
 k=\begin{pmatrix}\eta&\theta\\0&\lambda\end{pmatrix},\quad
 X=\begin{pmatrix}0&z\\0&v_0\end{pmatrix},\quad Y=re_{12},
\]
with $\bar\eta\lambda=\mu\in F^\times$,
$\theta/\lambda,r\in E^-$. The source and target Levi equations for
normalized points imply
\[
 \alpha=\lambda,\quad\epsilon=\eta,\quad
 \beta=v_0=0,\quad\zeta'=\zeta,\quad t=\eta/\lambda\in F^\times,
\]
\[
 \theta/\lambda=(x'-tx)/2,\qquad r=(y-ty')/2.
\]
The last two quantities also belong to $F$, hence vanish. This gives
\eqref{eq:padic-plane-transitions}. These relations hold also when $x=0$ or $y=0$.
The reductive characters cancel, and the two unipotent traces are $z$ and $-z$. By \eqref{eq:padic-plane-transitions}, the residual $F^\times$-action on the normalized family carries the trivial character. Moreover,
\begin{equation}\label{eq:padic-plane-transpose}
 g_\zeta(x,y)^t=g_\zeta(y,x).
\end{equation}

To apply the slice restriction statement of
Section~\ref{subsec:padic-tools}, we must check that the action map is
a surjective submersion and identify its fibers and density character.
The normal form in Section~\ref{sec:common3221} gives the surjective
map
\[
 \rho:\HH\times(E^\times\times F^2)\longrightarrow C\setminus C_0.
\]
Its differential also has full rank at the axes and at the origin. Linearizing the complete transition equations at the identity gives
\[
\dot\zeta=0,\qquad \dot x=\dot t\,x,\qquad \dot y=-\dot t\,y,
\]
with independent parameters $\dot t\in F$ and $\dot\lambda,\dot z\in E$. These give all the kernel directions, since the complete transition equations remain valid over the dual numbers and their solution uses only invertible pivots. Hence $\dim_F\ker d\rho=1+2+2=5$ everywhere. At $(x,y)=(0,0)$, the $\dot t$-direction becomes a stabilizer direction rather than a moving slice direction, so the kernel dimension does not increase. Since the source and target dimensions are $38$ and $33$, respectively, the differential is surjective at every point.

The fibers of $\rho$ are described by the group
\[
 L=F^\times\times E^\times\times E,
\]
acting by
\[
 (h_0,m)\longmapsto
 (h_0T(t,\lambda,z)^{-1},(\zeta,tx,t^{-1}y)).
\]
To check the group law despite the source-dependent notation, write
$m=(\zeta,x,y)$ and $a_tm=(\zeta,tx,t^{-1}y)$. Direct multiplication gives
\[
 j_m(z_1)j_m(z_2)=j_m(z_1+z_2),\qquad
 h(t,\lambda)j_m(z)h(t,\lambda)^{-1}=j_{a_tm}(z).
\]
Thus $T(l_2;a_{t_1}m)T(l_1;m)=T(l_2l_1;m)$, where the first two
coordinates multiply and the last adds. Also
\[
 \sigma(T(t,\lambda,z;m))=
       T(t^{-1},\lambda^{-1},-z;\sigma m).
\]
The action is free with fibers equal to its orbits. Use product Haar
densities $dh_0\,d^\times\zeta\,dx\,dy$ and
$d^\times t\,d^\times\lambda\,dz$.
The base Jacobian is $|t|_F|t|_F^{-1}=1$; the change in $h_0$ is a
right translation in the unimodular group $\HH$; the vertical group is
unimodular. Each pair $(s,v)$ in~\eqref{eq:padic-plane-transitions} has
$\chi(s,v)=1$. Transpose exchanges
$x,y$ and sends vertical parameters to $(t^{-1},\lambda^{-1},-z)$,
preserving these densities. The slice restriction statement says that pullback by such a surjective submersion is injective; the preceding character and density calculations show that no additional character appears. Hence restriction injects the odd equivariant distribution space into
\[
 \{T\in\mathcal D'(E^\times\times F^2):
       t_*T=T,\quad\sigma_*T=-T\},
\]
where $t(\zeta,x,y)=(\zeta,tx,t^{-1}y)$ and
$\sigma(\zeta,x,y)=(\zeta,y,x)$.

Lemma~\ref{lem:padic-plane}, followed by localization in the fixed $E^\times$-coordinate, gives vanishing on $C\setminus C_0$.

The last subset is $Z$. Here transpose is realized directly by an
element of character one.

Write $A=\left(\begin{smallmatrix}a&b\\0&d\end{smallmatrix}\right)$,
$B=\left(\begin{smallmatrix}m&0\\p&q\end{smallmatrix}\right)$ and put
\[
 \eta_0=m/a\in F^\times,\quad k=\diag(\eta_0,1),\quad
 v_0=\frac{qb}{ad}-\frac pm,\quad
 \beta=\frac ba-\frac{pd}{mq}.
\]
Thus $\beta q=dv_0$. Set
\[
 s=D(k)u\left(\begin{pmatrix}0&bv_0\\0&dv_0\end{pmatrix},0\right).
\]
This is an element of $S$. The identity
\begin{equation}\label{eq:padic-Ztransporter}
 s\,g(A,B)=g(A,B)^t s^t
\end{equation}
follows by multiplying the blocks. Alternatively, set
\[
 H_0=\begin{pmatrix}1&\beta\\0&d/q\end{pmatrix},\qquad
 h=m(H_0,k)u(v_0e_{22},0).
\]
Then use the two identities
\[
 kA=B^tH_0,\qquad
 \begin{pmatrix}1&0\\(d/q)v_0&d/q\end{pmatrix}B=A^tk^t.
\]
The corresponding factors $L_{B^t}hL_A^{-1}$ and
$(L_{A^t}^{-1}whwL_B)^t$ are both the displayed $s$.

Set $t=(s,s^{-1})\in\HH$. By \eqref{eq:padic-Ztransporter},
\[
t\cdot g(A,B)=s\,g(A,B)(s^{-1})^t=g(A,B)^t.
\]
Moreover,
\[
\chi(t)=\eta(s)\eta(s^{-1})=1,\qquad \sigma(t)=(s^{-1},s)=t^{-1}.
\]
Thus $t^{-1}\tau$ is an involution in the extended stabilizer of $g(A,B)$, with $\widetilde\chi(t^{-1}\tau)=-1$. For any $x\in Z$, choose $h\in\HH$ carrying $x$ to such a representative. Then $\sigma(h)^{-1}th$ carries $x$ to $\tau(x)$ and has the same character and inverse properties. Lemma~\ref{lem:padic-orbits} therefore gives
$\mathcal D'(Z)^{\widetilde\HH,\widetilde\chi}=0$. We have now proved
vanishing on all four successive pieces of
$G_{3221}\supset C\supset C_0\supset Z$. Applying
\eqref{eq:closed-support-tools} successively through this closed
filtration proves the proposition.
\end{proof}

\subsubsection{Vanishing on $G_{3111}$}

\begin{proposition}\label{prop:padic-3111}
Assume that $H$ is quasi-split. Then there are no nonzero odd
$(\HH,\chi)$-equivariant distributions on $G_{3111}$. Equivalently,
\[
 \mathcal D'(G_{3111})^{\widetilde\HH,\widetilde\chi}=0.
\]
\end{proposition}
\begin{proof}
The stabilizer character first reduces the problem to a closed subset
$C_5$. Points of $C_5$ are then normalized so that transpose is
realized by an element of character one.

Here $w=(2\ 5)$, so
\[
 w=\begin{pmatrix}e_{11}&0&e_{21}\\0&I_2&0\\e_{12}&0&e_{22}\end{pmatrix},
 \quad
 x(A,B)=\begin{pmatrix}
 Ae_{11}B&0&Ae_{21}B^*\\0&I_2&0\\A^*e_{12}B&0&A^*e_{22}B^*
 \end{pmatrix}.
\]
Write $A=[u,v]=\left(\begin{smallmatrix}a&b\\c&d\end{smallmatrix}\right)$
and $B=\left(\begin{smallmatrix}\alpha\\\beta\end{smallmatrix}\right)
=\left(\begin{smallmatrix}m&n\\p&q\end{smallmatrix}\right)$.
The full unipotent stabilizer at $w$ consists of
\[
 v_0=u\left(\begin{pmatrix}0&0\\r&s\end{pmatrix},\epsilon e_{21}\right),
 \qquad
 u_0=u\left(\begin{pmatrix}0&0\\-\bar s&-\bar r\end{pmatrix},
                    \epsilon e_{21}\right),
\]
where $r,s\in E$ and
$\epsilon+\bar\epsilon+r\bar s+s\bar r=0$.
The parameters $r,s$ are unrestricted; the remaining trace-zero part
of $\epsilon$ is free. Conjugation to $x(A,B)$ yields
\begin{equation}\label{eq:padic-3111char}
 \chi(u,v_1^{-1})=
 \psi\bigl(\Tr(- (\bar d+p)r-(\bar b+q)s)\bigr).
\end{equation}
Thus the stabilizer character is trivial exactly on the subset
\begin{equation}\label{eq:padic-3111compatible}
 C_5:\qquad\beta=-v^\dagger J,
 \quad\text{equivalently }p=-\bar d,\ q=-\bar b.
\end{equation}

The unipotent stabilizers have constant $F$-dimension five, so their Lie algebras form an algebraic vector bundle over $G_{3111}$. The additive functional underlying the character of $U\times U$ restricts to a section of its dual bundle. Since the stabilizers are connected unipotent groups in characteristic zero, the character is trivial precisely where this section vanishes. Thus $C_5$ is closed in $G_{3111}$. It is stable under $\HH$ and $\tau$ by normality of $U\times U$ and invariance of its character. For every point of $G_{3111}\setminus C_5$ the unipotent stabilizer
character is nontrivial. Lemma~\ref{lem:padic-orbits} therefore gives
vanishing on this open complement.

On $C_5$ two changes of representative, both effected by elements of character one, reduce the transpose problem to explicit matrices. Put $H(v)=v^\dagger Jv$. The first change uses triangular matrices. For
$g,j\in E^\times$ and $h,h_R\in E$, put
\[
 G_0=\begin{pmatrix}g&0\\h&j\end{pmatrix},\quad
 H_0=\begin{pmatrix}g&0\\h_R&1/\bar j\end{pmatrix},\qquad
 A'=AG_0^{-1},\quad B'=H_0^tB.
\]
This transition is effected by central unipotent factors and has
character one. Explicitly, with
\[
 Y_L=\diag(h_R/g,-\overline{h_R/g}),\qquad
 Y_R=\diag(h/g,-\overline{h/g}),
\]
one has
\[
 L_{G_0}u(0,Y_L)=w(L_{H_0}u(0,Y_R))^tw,
\]
and conjugation gives $u_Lx(A,B)=x(A',B')u_R^t$ with
\[
 u_L=L_Au(0,Y_L)L_A^{-1},\qquad
 u_R=L_{B'}^tu(0,Y_R)L_{B'}^{-t}.
\]
On columns and rows this operation is
\begin{equation}\label{eq:padic-3111triangular}
 v'=v/j,\quad u'=(u-(h/j)v)/g,\quad
 \beta'=\beta/\bar j,\quad\alpha'=g\alpha+h_R\beta.
\end{equation}
The second change is conjugation by $D(k)$. For $k\in K$ and $\mu=\lambda(k)$, direct block multiplication gives
\begin{equation}\label{eq:padic-3111K}
 D(k)x(A,B)D(k)^{-1}
 =x(kA\diag(1,\mu^{-1}),Bk^{-1}).
\end{equation}
The acting pair $(D(k),D(k^{-t}))$ has character one. Both operations
preserve \eqref{eq:padic-3111compatible}.

\paragraph{The case $H(v)\ne0$.}
Suppose $H(v)=v^\dagger Jv\ne0$. Set
\[
 C_v=\begin{pmatrix}\bar b&b\\-\bar d&d\end{pmatrix},\quad
 C_* =\begin{pmatrix}1&1\\-1&1\end{pmatrix},\quad k=C_*C_v^{-1}.
\]
Then $C_v^\dagger JC_v=\diag(-H(v),H(v))$, so
$k\in K$, $\lambda(k)=2/H(v)$ and $kv=(1,1)^t$.
Apply \eqref{eq:padic-3111K} and then
\eqref{eq:padic-3111triangular} with $j=H(v)/2$, $g=1$, $h=h_R=0$.
Now $v=v_*=(1,1)^t$ and $\beta=(-1,-1)$.
Write $u=z(1,-1)^t+t(1,1)^t$, where $z\ne0$.
Taking $j=1$, $g=z$, $h=t$, $h_R=z\alpha v_*/2$ gives
\[
 A=\begin{pmatrix}1&1\\-1&1\end{pmatrix},\qquad
 B=\begin{pmatrix}\lambda/2&-\lambda/2\\-1&-1\end{pmatrix},
 \qquad\lambda\in E^\times.
\]
The corresponding matrix is
\[
 y_\lambda=\begin{pmatrix}
 (\lambda/2)u_*u_*^t&0&-v_*v_*^t/2\\
 0&I_2&0\\-v_*v_*^t/2&0&u_*u_*^t/(2\bar\lambda)
 \end{pmatrix},\qquad u_*=(1,-1)^t,
\]
and is symmetric. 

\paragraph{The case $H(v)=0$.}
Suppose $H(v)=0$, still with $v\ne0$. Choose a standard basis vector
$e_0$ such that $c_0=e_0^\dagger Jv\ne0$, and set
\[
 e=e_0/\bar c_0,\qquad e'=e-(e^\dagger Je)v/2,\qquad
 C=[e',v].
\]
Then $C^\dagger JC=J$, so $k=C^{-1}\in K$ sends $v$ to $e_2$.
After \eqref{eq:padic-3111K}, write
$u=(a,c)^t$, $\alpha=(m,n)$ and $\beta=(-1,0)$, with $a,n\ne0$.
Taking $j=1$, $g=a$, $h=c$, $h_R=am$ in
\eqref{eq:padic-3111triangular} gives
\[
 A=I_2,\qquad B=\begin{pmatrix}0&\rho\\-1&0\end{pmatrix},\quad
 \rho=an\in E^\times,
\]
and the matrix
\[
 z_\rho=\begin{pmatrix}
 \rho e_{12}&0&-e_{22}\\0&I_2&0\\-e_{11}&0&\bar\rho^{-1}e_{21}
 \end{pmatrix}.
\]
It satisfies $D(J)z_\rho D(J)^{-1}=z_\rho^t$.
Here $t=(D(J),D(J))$ has character
$\omega(J)^2=\chi_E(-1)^2=1$ and satisfies $\sigma(t)=t^{-1}$.

For either normalized representative $y=y_\lambda$ or $y=z_\rho$, let $t_y=1$ in the nonisotropic case and $t_y=(D(J),D(J))$ in the isotropic case. The preceding calculations give
\[
t_y\cdot y=\tau(y),\qquad \chi(t_y)=1,\qquad \sigma(t_y)=t_y^{-1}.
\]
Now let $x\in C_5$. By the two normalizations above, there exist $h\in\HH$ and a normalized representative $y$ such that $h\cdot x=y$. Set $t_x=\sigma(h)^{-1}t_yh$. Since $\tau(h\cdot x)=\sigma(h)\cdot\tau(x)$ and $\chi\circ\sigma=\chi$, we obtain
\[
t_x\cdot x=\tau(x),\qquad \chi(t_x)=1,\qquad \sigma(t_x)=t_x^{-1}.
\]
Consequently $t_x^{-1}\tau$ is an element of the extended
stabilizer of $x$ with $\widetilde\chi(t_x^{-1}\tau)=-1$.
Lemma~\ref{lem:padic-orbits} states that this stabilizer condition
rules out a nonzero odd equivariant distribution on the orbit and,
for the algebraic action considered here, on the whole invariant
subset. Hence
\[
\mathcal D'(C_5)^{\widetilde\HH,\widetilde\chi}=0.
\]
Since $C_5$ is closed and invariant and we already have vanishing on
its complement, \eqref{eq:closed-support-tools} gives the proposition.
\end{proof}

\subsubsection{Conclusion for the quasi-split group}

Propositions~\ref{prop:open}, \ref{prop:padic-2000}, \ref{prop:padic-2110}, \ref{prop:padic-3221}, and \ref{prop:padic-3111} establish vanishing on all five strata. The required closed filtration is obtained by putting
\[
 b(g)=\rank g_{\{3,4,5,6\},\{5,6\}},\qquad
 d(g)=\rank g_{\{5,6\},\{5,6\}}.
\]
The closed filtration
\[
 G\supset\{d\le1\}\supset\{b\le1\}
 \supset\{d=0\}\supset\{b=0\}\supset\varnothing
\]
has successive strata $4222,3221,3111,2110,2000$.
Equation~\eqref{eq:closed-support-tools} says that a distribution
supported on the next closed subset is the pushforward of a
distribution on that subset. Applying it successively along this
filtration gives
\[
 \mathcal D'(G)^{\widetilde\HH,\widetilde\chi}=0.
\]
This proves the non-Archimedean assertion of
Theorem~\ref{thm:distribution} for the quasi-split group.

\subsection{The non-quasi-split group}\label{sec:nonsplit-padic}

Here $H=\diag(1,d)$ with $-d\notin N(E^\times)$, and
$\tau(g)=C^{-1}g^tC$ as in~\eqref{eq:tau-common}. By
Propositions~\ref{prop:open} and~\ref{prop:padic-2000}, only
$G_{3111}$ remains.

\begin{proposition}\label{prop:padic-nonsplit-3111}
Assume that $H$ is non-quasi-split. Then there are no nonzero odd
$(\HH,\chi)$-equivariant distributions on $G_{3111}$. Equivalently,
\[
 \mathcal D'(G_{3111})^{\widetilde\HH,\widetilde\chi}=0.
\]
\end{proposition}
\begin{proof}
A unipotent stabilizer character first removes an open part of the
stratum. The remaining points are moved, by elements of character one,
to representatives fixed by $\tau$.

Let $w=(2\ 6)$, viewed as a permutation matrix.  The restriction of $W$ to the second
and sixth coordinate plane is $\bigl(\begin{smallmatrix}0&d\\d&0\end{smallmatrix}\bigr)$, so
$w\in G$, $w^2=I_6$, and $\tau(w)=w$.
Every $\HH$-orbit in this stratum has a representative
\begin{equation}\label{eq:ns-padic-3111-cover}
 x(A,B)=L_AwL_B,
 \qquad
 A=\begin{pmatrix}a&b\\c&e\end{pmatrix}=[u,v],
 \qquad B=\begin{pmatrix}m&n\\p&q\end{pmatrix}
                   \quad (A,B\in\GL_2(E)).
\end{equation}

For arbitrary $r,s\in E$, choose $\epsilon\in E$ with
\[
           \epsilon+\bar\epsilon=-d\,r\bar r-s\bar s,
\]
which is possible since $\operatorname{char}F=0$.  Set
\[
 v_0=u\!\left(\begin{pmatrix}0&0\\r&s\end{pmatrix},
                    \epsilon e_{22}\right),\qquad
 u_0=u\!\left(\begin{pmatrix}0&0\\-\bar r&-\bar s\end{pmatrix},
                    \epsilon e_{22}\right).
\]
Both matrices belong to $U$, and multiplication gives
$u_0w=w\tau(v_0)$.  Consequently
\[
 u_A=L_Au_0L_A^{-1},\qquad
 v_B=\tau(L_B)v_0\tau(L_B)^{-1}
\]
belong to $U$ and satisfy $u_Ax(A,B)=x(A,B)\tau(v_B)$.
Their character is
\begin{equation}\label{eq:ns-padic-3111-character}
 \chi(u_A,v_B^{-1})=
 \psi\!\left(\operatorname{Tr}_{E/F}
        \bigl[-(\bar b+dp)r-(\bar e+q)s\bigr]\right).
\end{equation}
Indeed, conjugation by $L_A$ sends the $(1,2)$ block of $u_0$
to $A(X_{u_0})$, whose trace is $-b\bar r-e\bar s$.
Conjugation by $\tau(L_B)$ sends the corresponding block of $v_0$
to $\tau_2(B)(X_{v_0})$, whose trace is $dp\,r+q\,s$.
Taking the field trace yields
\eqref{eq:ns-padic-3111-character}.

The character obstruction in Lemma~\ref{lem:padic-orbits} applies unless
\begin{equation}\label{eq:ns-padic-compatible}
     p=-\bar b/d,\qquad q=-\bar e,
     \qquad\text{or equivalently}\qquad
     (p,q)=-d^{-1}v^\dagger H.
\end{equation}
Thus it remains to treat the closed subset defined by this equation. For each of its points we will find a point in the same $\HH$-orbit that is fixed by $\tau$.

The remaining points can be simplified without changing the character.

The required changes of parameters come from the center of $U$ and do not change the character. For
$\gamma,j\in E^\times$ and $h,h_R\in E$ define
\begin{align*}
 G_0&=\begin{pmatrix}\gamma&0\\h&j\end{pmatrix},&
 H_0&=\begin{pmatrix}\gamma&0\\h_R&\bar j^{-1}\end{pmatrix},\\
 Q_L&=\begin{pmatrix}0&dh_R/\gamma\\
                -\bar h_R/\bar\gamma&0\end{pmatrix},&
 Q_R&=\begin{pmatrix}0&dh/\gamma\\
                -\bar h/\bar\gamma&0\end{pmatrix},\\
 U_L&=u(0,Q_L),& U_R&=u(0,Q_R).
\end{align*}
The matrices $Q_L,Q_R$ satisfy $HQ+Q^\dagger H=0$, so
$U_L,U_R$ lie in the center of $U$.  Direct multiplication gives
\begin{equation}\label{eq:ns-padic-central-intersection}
                 L_{G_0}U_L=w\tau(L_{H_0}U_R)w.
\end{equation}
For
\begin{equation}\label{eq:ns-padic-parameter-change}
             A'=AG_0^{-1},\qquad B'=\tau_2(H_0)B,
\end{equation}
put
\[
 s_L=L_AU_LL_A^{-1},\qquad
 s_R=\tau(L_{B'})U_R\tau(L_{B'})^{-1}.
\]

Both $s_L$ and $s_R$ are conjugates of central unipotent elements of $U$, and hence have character one. Since $L_{A'}=L_A L_{G_0}^{-1}$ and $L_{B'}=\tau(L_{H_0})L_B$, equation~\eqref{eq:ns-padic-central-intersection} gives
\begin{equation}\label{eq:ns-padic-central-transport}
s_Lx(A,B)=L_AU_LwL_B=L_{A'}w\tau(U_R)L_{B'}=x(A',B')\tau(s_R).
\end{equation}
Thus $(s_L,s_R^{-1})$ sends $x(A,B)$ to $x(A',B')$.

These changes produce representatives fixed by $\tau$.

Suppose \eqref{eq:ns-padic-compatible} holds, and put
\[
 h(v)=v^\dagger Hv=b\bar b+d e\bar e.
\]
As $v\neq0$ and $H$ is anisotropic, $h(v)\neq0$.  The matrices
\begin{equation}\label{eq:ns-padic-k-normalizer}
 C_v=\begin{pmatrix}-d\bar e&b\\\bar b&e\end{pmatrix},
 \qquad k=\diag(-d,1)C_v^{-1}
\end{equation}
satisfy
\[
 C_v^\dagger HC_v=h(v)\diag(d,1),\qquad
 kv=e_2,\qquad k^\dagger Hk=\ell H,
 \qquad \ell=\frac d{h(v)}.
\]
Thus $k\in K$ with multiplier $\ell$.

Conjugating $x(A,B)$ by $D(k)$ gives the new parameters
\begin{equation}\label{eq:ns-padic-diagonal-transport}
 A_1=kA\diag(1,\ell^{-1}),\qquad B_1=Bk^{-1},
 \qquad
 x(A_1,B_1)=D(k)x(A,B)D(k)^{-1}.
\end{equation}

Conjugation by $D(k)$ is effected by $(D(k),\tau(D(k)^{-1}))\in\HH$, whose character is
\[
\omega(k)\omega(\tau_2(k^{-1}))=1.
\]
Put $\beta=(p,q)=-d^{-1}v^\dagger H$. Since $k^\dagger Hk=\ell H$, we have $Hk^{-1}=\ell^{-1}k^\dagger H$. Thus the second column of $A_1$ and the second row of $B_1$ are
\[
(A_1)_{\bullet2}=\ell^{-1}kv=\ell^{-1}e_2,\qquad
(B_1)_{2\bullet}=\beta k^{-1}=-\frac{1}{d\ell}(kv)^\dagger H=(0,-\ell^{-1}).
\]

The first parameter change is obtained by applying \eqref{eq:ns-padic-parameter-change} with $\gamma=1$, $h=h_R=0$, and $j=\ell^{-1}$. Since $\ell\in F^\times$, this gives $G_0=\diag(1,\ell^{-1})$ and $\tau_2(H_0)=\diag(1,\ell)$. Consequently,
\[
A_2=A_1\diag(1,\ell),\qquad B_2=\diag(1,\ell)B_1.
\]
The second column of $A_2$ is $e_2$, while the second row of $B_2$ is $(0,-1)$. Since both matrices are invertible, they have the form
\[
A_2=\begin{pmatrix}a_0&0\\c_0&1\end{pmatrix},\qquad
B_2=\begin{pmatrix}m_0&n_0\\0&-1\end{pmatrix},\qquad a_0m_0\ne0.
\]
A second application uses with $\gamma=a_0$, $h=c_0$, $j=1$, and $h_R=a_0n_0/d$. Then
\[
G_0=A_2,\qquad \tau_2(H_0)=\begin{pmatrix}a_0&a_0n_0\\0&1\end{pmatrix}.
\]
It follows that $A_3=A_2G_0^{-1}=I_2$ and
\[
B_3=\begin{pmatrix}a_0&a_0n_0\\0&1\end{pmatrix}
\begin{pmatrix}m_0&n_0\\0&-1\end{pmatrix}
=\begin{pmatrix}a_0m_0&0\\0&-1\end{pmatrix}.
\]
Writing $\rho=a_0m_0\in E^\times$, we obtain $A_3=I_2$ and $B_3=\diag(\rho,-1)$.
The representative is therefore
\begin{equation}\label{eq:ns-padic-fixed-representative}
 x_\rho=
 \begin{pmatrix}
 \rho e_{11}&0&-e_{22}\\
 0&I_2&0\\
 -e_{22}&0&\bar\rho^{-1}e_{11}
 \end{pmatrix},\qquad \tau(x_\rho)=x_\rho.
\end{equation}
Thus every orbit satisfying~\eqref{eq:ns-padic-compatible} contains
a point fixed by $\tau$. At such a point $\tau$ belongs to the
extended stabilizer and $\widetilde\chi(\tau)=-1$.
Lemma~\ref{lem:padic-orbits} therefore gives vanishing on the closed
subset defined by~\eqref{eq:ns-padic-compatible}; the same lemma has
already given vanishing on its complement. Equation
\eqref{eq:closed-support-tools} proves the proposition.
\end{proof}

Together with Propositions~\ref{prop:open},
\ref{prop:padic-2000}, and~\ref{prop:padic-nonsplit-3111}, the closed
filtration from Section~\ref{sec:orbits} gives
\[
 \mathcal D'(G)^{\widetilde\HH,\widetilde\chi}=0
\]
in the non-quasi-split case as well.

\subsection{The multiplicity bound}

For either Hermitian plane we have proved
$\mathcal D'(G)^{\widetilde\HH,\widetilde\chi}=0$. If $T$ is
$(\HH,\chi)$-equivariant, then $T-\tau_*T$ is odd and must vanish;
hence every such distribution is $\tau$-invariant.

The Gelfand--Kazhdan criterion in
Section~\ref{subsec:padic-tools} states that this distributional
symmetry implies, for every irreducible admissible smooth
representation $\pi$,
\[
 \dim\Hom_S(\pi,\eta)\,
 \dim\Hom_{\tau(S)}(\pi^\vee,\eta\circ\alpha)\leq1,
 \qquad\alpha(g)=\tau(g)^{-1}.
\]
Lemma~\ref{lem:dual} states that $\pi^\vee\simeq\pi\circ\alpha$, and
$\alpha(\tau(S))=S$. Thus
\[
 \Hom_{\tau(S)}(\pi^\vee,\eta\circ\alpha)
 \simeq\Hom_S(\pi,\eta).
\]
The square of the desired dimension is at most one. This proves
Theorem~\ref{thm:main} over every non-Archimedean field of
characteristic zero, for both forms. No assumption on the residue
characteristic was used.

\section{Multiplicity one over an Archimedean local field}\label{sec:real}

Throughout this section $E/F=\C/\R$, $H\in\{J,I_2\}$ and $\tau(g)=g^t$.
We write $\psi(t)=e^{i\kappa t}$ with $\kappa\in\R^\times$, and use the
notation and the bi-invariant metric of Section~\ref{subsec:real-tools}.
All tangent and normal spaces are real.

To prove Theorem~\ref{thm:distribution} over $\R$, we must show that an
odd $(\HH,\chi)$-equivariant tempered generalized function satisfying a
Casimir eigenvalue equation vanishes. Proposition~\ref{prop:open}
already gives vanishing on the open cell $O$. To exclude support on the
boundary, including normal derivatives along its strata, we use the
support lemmas of Section~\ref{subsubsec:real-support}.

In Section~\ref{subsec:casimir}, we express the Casimir operator, on
equivariant generalized functions, as a nonzero multiple of a transverse
vector field plus a tangential operator. Lemma~\ref{lem:transverse} then excludes support on
$G_{3111}$ and $G_{2000}$ in both cases, and on $G_{2110}$ in the
quasi-split case. The remaining quasi-split stratum $G_{3221}$ is treated
in Section~\ref{subsec:real-qs-boundary}. For the filtration constructed
in Section~\ref{sec:common3221}, we exhibit either unipotent stabilizers
with nontrivial character or normal vectors of nonzero length on each
successive difference. This verifies the hypotheses of
Lemma~\ref{lem:real-filtration} and completes the boundary argument.

\subsection{The Casimir operator and the lower strata}\label{subsec:casimir}

For $Y\in U(\g_\C)$ let $Y^L$ be the left-invariant differential
operator of Section~\ref{subsec:real-tools}. Thus $Y^L_g=gY$ for
$Y\in\g$, and $(YY')^L=Y^LY'^L$. Put
\[
 X_0=N(I_2,0)=\begin{pmatrix}0&I_2&0\\0&0&-I_2\\0&0&0\end{pmatrix}\in\uu,
 \qquad\mathcal X=X_0^L,
 \qquad H_\rho=\diag(-4I_2,0,4I_2)\in\mm.
\]
Thus $\mathcal X_g=gX_0$.

\begin{lemma}\label{lem:casimir}
Put $D_M=\Omega_M+H_\rho\in U(\mm_\C)$. If a generalized function $f$ on
$G$ satisfies $f(g\tau(u))=\xi(u)f(g)$ for all $u\in U$, then
\begin{equation}\label{eq:casimir}
 \Omega f=2i\kappa\,\mathcal Xf+D_M^Lf.
\end{equation}
In particular, \eqref{eq:casimir} holds for every generalized function
satisfying~\eqref{eq:equiv}.
\end{lemma}
\begin{proof}
Let $\g=\bigoplus_{k=-2}^{2}\g_k$ be the eigenspace decomposition of
$\operatorname{ad}\diag(I_2,0,-I_2)$; note that $\diag(I_2,0,-I_2)\in\mm$.
Then $\g_0=\mm$, $\uu=\g_1\oplus\g_2$ and $\uu^t=\g_{-1}\oplus\g_{-2}$.
By invariance of $B$, $B(\g_k,\g_l)=0$ unless $k+l=0$. Hence $\uu$ and
$\uu^t$ are isotropic, $\mm$ is orthogonal to both, and $B$ restricts to
nondegenerate pairings of $\uu$ with $\uu^t$ and of $\mm$ with itself.

Choose a basis $(E_j)$ of $\uu$ with each $E_j$ in $\g_1$ or $\g_2$. Let
$(F_j)$ be the $B$-dual basis of $\uu^t$, and let $(M_a)$, $(M^a)$ be
$B$-dual bases of $\mm$, so that $\Omega_M=\sum_aM_aM^a$. Then
$(E_j,M_a,F_j)$ and $(F_j,M^a,E_j)$ are $B$-dual bases of $\g$, and
\begin{equation}\label{eq:casimir-decomposition}
 \Omega=\Omega_M+\sum_j(E_jF_j+F_jE_j)
       =2\sum_jE_jF_j+\Omega_M+\sum_j[F_j,E_j].
\end{equation}
Since $B$ pairs $\g_k$ only with $\g_{-k}$, we have $F_j\in\g_{-k}$
whenever $E_j\in\g_k$; hence $[F_j,E_j]\in\g_0=\mm$ for every $j$. To
identify $\sum_j[F_j,E_j]$, let $Y\in\mm$.
Invariance of $B$ gives
\[
 B\Bigl(Y,\sum_j[F_j,E_j]\Bigr)=\sum_jB([Y,F_j],E_j)
 =\tr\bigl(\operatorname{ad}Y|_{\uu^t}\bigr)
 =-\tr\bigl(\operatorname{ad}Y|_{\uu}\bigr),
\]
where the last equality holds because $B$ identifies $\uu^t$ with the
dual of $\uu$ as an $\mm$-module. Write
$Y=\diag(A,Z_1,\ell(Z_1)I_2-H^{-1}A^\dagger H)$ with $Z_1\in\kk$, and
note that $\Ree\tr Z_1=\ell(Z_1)$. On $\g_1\simeq M_2(\C)$ the operator
$\operatorname{ad}Y$ is $X\mapsto AX-XZ_1$. On $\g_2$ it is
$C'\mapsto AC'-C'(\ell(Z_1)I_2-H^{-1}A^\dagger H)$. Both have real trace
$4\Ree\tr A-4\ell(Z_1)$. Hence
\[
 \tr\bigl(\operatorname{ad}Y|_{\uu}\bigr)=8\Ree\tr A-8\ell(Z_1)
 =-B(Y,H_\rho).
\]
As $B$ is nondegenerate on $\mm$, this shows $\sum_j[F_j,E_j]=H_\rho$,
and therefore $\Omega=2\sum_jE_jF_j+D_M$.

Now let $F\in\uu^t$, and write $F=N^t$, where $N\in\uu$ has upper
adjacent block $A$. The element $\exp(tN)\in U$ has upper adjacent block
$tA$, so the hypothesis gives
\[
 f(g\exp(tF))=f\bigl(g\tau(\exp tN)\bigr)
 =\psi\bigl(\Tr(\tr tA)\bigr)f(g)=e^{2i\kappa t\Ree\tr A}f(g).
\]
Direct multiplication gives
$B(N,X_0^t)=\Ree\tr A+\Ree\tr(H^{-1}A^\dagger H)=2\Ree\tr A$, and $B$ is
invariant under transpose. Therefore
\[
 F^Lf=i\kappa\,B(F,X_0)\,f\qquad(F\in\uu^t).
\]
Since $X_0=\sum_jB(X_0,F_j)E_j$, we obtain
\[
 \sum_j(E_jF_j)^Lf=\sum_jE_j^L\bigl(i\kappa B(F_j,X_0)f\bigr)
 =i\kappa X_0^Lf=i\kappa\,\mathcal Xf.
\]
Applying $\Omega=2\sum_jE_jF_j+D_M$ to $f$ gives~\eqref{eq:casimir}.
Throughout, derivatives of generalized functions are defined by duality.
\end{proof}

\begin{lemma}\label{lem:casimir-transverse}
The operator $D_M^L$ is tangential to every rank stratum. The vector
field $\mathcal X$ is transverse to the strata $G_{3111}$, $G_{2110}$
and $G_{2000}$ (the second is empty if $H=I_2$); that is,
$\mathcal X_g\notin T_gG_R$ for every $g$ in each of these strata $G_R$.
\end{lemma}
\begin{proof}
For $Y\in\mm$ the flow of $Y^L$ is right multiplication by
$\exp(tY)\in M\subset\tau(P)=P^t$. It preserves every $P$--$\tau(P)$
double coset, that is, every rank stratum (Lemma~\ref{lem:rank-common}).
Hence $Y^L$ is tangent to every stratum, and $D_M^L$, a polynomial in
such vector fields with constant coefficients, is tangential to every
stratum.

Write $g=(g_{ij})_{1\le i,j\le3}$ in $2\times2$ blocks, so that
$r_{21}=\rank(g_{32},g_{33})$ and $r_{22}=\rank g_{33}$. Since
$\mathcal X_g=gX_0$, along the integral curve $g\exp(tX_0)$ the bottom
block row moves with velocity
\begin{equation}\label{eq:casimir-bottom-derivative}
 (\dot g_{31},\dot g_{32},\dot g_{33})=(0,\,g_{31},\,-g_{32}).
\end{equation}
Suppose that $\mathcal X_g$ is tangent to the stratum through $g$.

On $G_{2110}$ we have $g_{33}=0$ and $\rank g_{32}=1$. Since $g_{33}$
vanishes identically on the stratum, tangency gives
$\dot g_{33}=-g_{32}=0$, a contradiction.

On $G_{2000}$ we have $g_{32}=g_{33}=0$. Tangency gives
$\dot g_{32}=g_{31}=0$. Then the bottom block row of $g$ vanishes,
contradicting invertibility.

On $G_{3111}$ the $2\times4$ matrix $\Lambda=(g_{32},g_{33})$ has rank
one, and so does $g_{33}$. Hence
$\operatorname{im}g_{32}\subset\operatorname{im}\Lambda
=\operatorname{im}g_{33}$. Choose a nonzero row vector $\varpi$ with
$\varpi\Lambda=0$. If $\Lambda(t)$ is a smooth curve of matrices of
constant rank, every $v\in\ker\Lambda(0)$ extends to a smooth curve
$v(t)\in\ker\Lambda(t)$; differentiating $\Lambda(t)v(t)=0$ at $t=0$
gives
\begin{equation}\label{eq:casimir-fixed-rank-tangent}
 \dot\Lambda(\ker\Lambda)\subset\operatorname{im}\Lambda.
\end{equation}
The rank of $\Lambda$ is constant on the stratum. By
\eqref{eq:casimir-bottom-derivative}, tangency therefore gives
$\varpi(g_{31}x-g_{32}y)=0$ whenever $g_{32}x+g_{33}y=0$. For every
$x\in\C^2$ such a $y$ exists, because
$\operatorname{im}g_{32}\subset\operatorname{im}g_{33}$; as
$\varpi g_{32}=0$, we get $\varpi g_{31}x=0$ for all $x$. Thus
$\varpi(g_{31},g_{32},g_{33})=0$, again contradicting invertibility of
$g$.
\end{proof}

\begin{proposition}\label{prop:casimir-extension}
Let $f$ be a generalized function on $G$ satisfying~\eqref{eq:equiv} and
$\Omega f=\gamma f$ for some $\gamma\in\C$.
\begin{enumerate}
\item If $H=J$ and $f|_{G'}=0$, then $f=0$.
\item If $H=I_2$ and $f|_O=0$, then $f=0$.
\end{enumerate}
\end{proposition}
\begin{proof}
By Lemma~\ref{lem:casimir},
\[
 (\mathcal X+D_\gamma)f=0,\qquad
 D_\gamma=\frac{D_M^L-\gamma}{2i\kappa},
\]
and by Lemma~\ref{lem:casimir-transverse} the operator $D_\gamma$ is
tangential to every rank stratum.

We apply Lemma~\ref{lem:transverse}, the first-order case
of~\cite[Lemma~2.2]{JSZ}, with $V=\mathcal X$ and $D=D_\gamma$.
For $H=J$, consider the open sets
$G'\subset G_1\subset G_2\subset G$ of~\eqref{eq:open-chain}. Their
successive differences $G_{3111}$, $G_{2110}$ and $G_{2000}$ are closed
submanifolds of $G_1$, $G_2$ and $G$, respectively
(Lemma~\ref{lem:rank-common} and the discussion
after~\eqref{eq:open-chain}). If $f|_{G'}=0$, then $f|_{G_1}$ is
supported in $G_{3111}$, and Lemmas~\ref{lem:transverse}
and~\ref{lem:casimir-transverse} give $f|_{G_1}=0$. In the same way
$f|_{G_2}=0$, and then $f=0$.

Let $H=I_2$. The chain is $O\subset\{r_{22}\ge1\}\subset G$, with
differences $G_{3111}$ and $G_{2000}$, and the same argument applies.
\end{proof}

\subsection{The quasi-split boundary}\label{subsec:real-qs-boundary}

Assume $H=J$. The complement of $O$ in $G'$ is $T=G_{3221}$. By
Section~\ref{sec:common3221} it has the filtration
\[
 T\supset C\supset C_0\supset Z\supset\varnothing
\]
defined by~\eqref{eq:C}--\eqref{eq:Z}. Its members are $\HH$-stable and
closed in $G'$, and its successive differences are embedded submanifolds
of $G$. We verify the hypotheses of Lemma~\ref{lem:real-filtration} on
these four differences. The first and third are unipotently
$\chi$-incompatible; the second and fourth are metrically proper.
We first explain how to check these properties on the families of
representatives from Section~\ref{sec:common3221}.

Let $Y\subset G$ be an $\HH$-stable submanifold. A
\emph{parametrization} of $Y$ is a smooth map $p\mapsto g_p$ from a
manifold $\mathcal P$ to $Y$ such that
\begin{equation}\label{eq:parametrization}
 \HH\times\mathcal P\longrightarrow Y,\qquad(h,p)\longmapsto h\cdot g_p,
\end{equation}
is a surjective submersion. Section~\ref{sec:common3221} provides one for
each layer. For $T\setminus C$ and $C_0\setminus Z$ we restrict the
covering~\eqref{eq:cover} to the parameters $(A,B)$ with $c\ne n$, and
with $c=n=0$ and $dq/a$, $dq/m$ not both real, respectively. The
products of $\HH$ with these parameter sets are the preimages of the
layers under~\eqref{eq:cover}. A submersion restricts to a submersion
from the preimage of a submanifold onto that submanifold, so the
restricted maps are again surjective submersions.
For $C\setminus C_0$ we use the normalized family~\eqref{eq:normalized},
and for $Z$ the diagonal family~\eqref{eq:diagonalAB}.

\begin{lemma}\label{lem:unipotent-criterion}
Let $p\mapsto g_p$ be a parametrization of an $\HH$-stable submanifold
$Y\subset G$. Suppose that there are smooth maps $s,v:\mathcal P\to S$
such that, for every $p\in\mathcal P$, the elements $s_p$ and $v_p$ are
unipotent, $s_pg_p=g_pv_p^t$, and $\chi(s_p,v_p^{-1})\ne1$. Then $Y$ is
unipotently $\chi$-incompatible.
\end{lemma}
\begin{proof}
Put $a_p=(s_p,v_p^{-1})\in\HH$. Then $a_p\cdot g_p=s_pg_pv_p^{-t}=g_p$,
and the differential at $g_p$ of $x\mapsto a_p\cdot x$ is
\begin{equation}\label{eq:boundary-normal-action}
 g_pX\longmapsto s_pg_pXv_p^{-t}=g_p\Ad(v_p^t)X\qquad(X\in\g).
\end{equation}
This is unipotent because $v_p^t$ is. Hence $a_p$ induces a unipotent
endomorphism of $T_{g_p}G/T_{g_p}Y$.

Let $y_0\in Y$. Since~\eqref{eq:parametrization} is a surjective
submersion, it has a smooth local section $y\mapsto(h(y),p(y))$ on an
open neighborhood $Y_0$ of $y_0$ in $Y$. Put
\[
 \phi(y)=h(y)\,a_{p(y)}\,h(y)^{-1}\qquad(y\in Y_0).
\]
Then $\phi$ is smooth, $\phi(y)\cdot y=y$, and
$\chi(\phi(y))=\chi(a_{p(y)})\ne1$. The differential of $h(y)$ carries
$T_{g_{p(y)}}Y$ to $T_yY$ and conjugates the endomorphism induced by
$a_{p(y)}$ into the one induced by $\phi(y)$, which is therefore
unipotent. Since $Y_0$ is open in $Y$, the action map
$\HH\times Y_0\to Y$ is a submersion. Thus $Y_0$ is a local slice, and $\phi$ satisfies all three
conditions of Definition~\ref{def:unipotent-incompatible}
\cite[Definition~3.3]{JSZ}.
\end{proof}

For the metric calculations, let $Q=\diag(Q_1,Q_2,Q_3)\in\g$ in
$2\times2$ blocks and suppose $Q_1+Q_2+Q_3=0$. Since
$\s=D_*(\kk)\oplus\uu$, we have
\begin{equation}\label{eq:block-orthogonality}
 B(Q,S')=B(Q,S'^t)=0\qquad(S'\in\s).
\end{equation}
Indeed, for $Z_1\in\kk$ and $N\in\uu$ we have
$B(Q,D_*(Z_1)+N)=\Ree\tr\bigl((Q_1+Q_2+Q_3)Z_1\bigr)=0$, and similarly
for transposes.

\begin{lemma}\label{lem:metric-criterion}
Let $p\mapsto g_p$ be a parametrization of an $\HH$-stable submanifold
$Y\subset G$. Suppose that for every $p\in\mathcal P$ there is $Q\in\g$
such that
\begin{enumerate}
\item both $Q$ and $g_pQg_p^{-1}$ are block diagonal in $2\times2$
blocks, and the three diagonal blocks of each matrix sum to zero;
\item $B(Q,g_p^{-1}\dot g)=0$ for every tangent vector $\dot g$ at $g_p$
to the family $p\mapsto g_p$;
\item $B(Q,Q)\ne0$.
\end{enumerate}
Then $Y$ is metrically proper.
\end{lemma}
\begin{proof}
Fix $p$, put $g=g_p$, and consider $gQ\in T_gG$. Its squared length is
$\langle gQ,gQ\rangle_g=B(Q,Q)\ne0$. The differential
of~\eqref{eq:parametrization} at $(1,p)$ maps onto $T_gY$, so $T_gY$ is
spanned by the orbit directions $S_1g+gS_2^t$, with $S_1,S_2\in\s$, and
by the vectors $\dot g$ of~(2). By invariance of $B$
and~\eqref{eq:block-orthogonality},
\[
 \langle gQ,S_1g+gS_2^t\rangle_g=B(gQg^{-1},S_1)+B(Q,S_2^t)=0,
\]
and $\langle gQ,\dot g\rangle_g=B(Q,g^{-1}\dot g)=0$ by~(2). Hence
$gQ\perp T_gY$. Its orthogonal complement is a proper nondegenerate
subspace containing $T_gY$, so $Y$ is metrically proper at $g$ by
Definition~\ref{def:metric-proper}. Every point of $Y$
has the form $h\cdot g_p$, and each $h\in\HH$ acts on $G$ by a left
translation followed by a right translation. This is an isometry of the
bi-invariant metric, and it preserves $Y$. Thus $Y$ is metrically proper
at every point.
\end{proof}

We now treat the four layers, using the notation~\eqref{eq:AB}.

\begin{lemma}\label{lem:firstU}
The layer $T\setminus C$ is unipotently $\chi$-incompatible.
\end{lemma}
\begin{proof}
Parametrize $T\setminus C$ by $g(A,B)$ with $(A,B)\in\GL_2(\C)^2$ and
$c\ne n$. For $z\in\C$ put
\[
 s=L_Au(ze_{12},0)L_A^{-1},\qquad v=L_{B^t}u(ze_{12},0)L_{B^t}^{-1}.
\]
These are unipotent elements of $U\subset S$. Since $L_B^t=L_{B^t}$, the
stabilizer calculation preceding~\eqref{eq:firstchar} gives
$sg(A,B)=g(A,B)v^t$, and~\eqref{eq:firstchar} gives
\[
 \chi(s,v^{-1})=\exp\bigl(i\kappa\Tr((c-n)z)\bigr).
\]
The choice $z=\pi\,\overline{(c-n)}/(2\kappa|c-n|^2)$ depends smoothly on
$(A,B)$ and makes this character equal to $e^{i\pi}=-1$.
Lemma~\ref{lem:unipotent-criterion} applies.
\end{proof}

\begin{lemma}\label{lem:Cmetric}
The layer $C^\times=C\setminus C_0$ is metrically proper.
\end{lemma}
\begin{proof}
Parametrize $C^\times$ by the normalized matrices $g=g_\zeta(x,y)$
of~\eqref{eq:normalizedmatrix}, with $(\zeta,x,y)\in\C^\times\times\R^2$.
Put
\begin{equation}\label{eq:boundary-Qy}
 Q_y=\diag(0,1,1,-1,-1,0)+\frac y2(E_{21}-E_{65}).
\end{equation}
By~\eqref{eq:lie-algebra}, $Q_y\in\g$, with $\ell(Q_y)=0$. Its diagonal
blocks are
$\left(\begin{smallmatrix}0&0\\y/2&1\end{smallmatrix}\right)$,
$\diag(1,-1)$ and
$\left(\begin{smallmatrix}-1&0\\-y/2&0\end{smallmatrix}\right)$. Thus
the three diagonal blocks sum to zero, and $B(Q_y,Q_y)=1+2+1=4$.
From~\eqref{eq:normalizedmatrix} one computes
\begin{equation}\label{eq:Qconjugate}
 gQ_yg^{-1}=-Q_x^t,
\end{equation}
so $gQ_yg^{-1}$ is also block diagonal with its three blocks summing
to zero. Moreover
\[
 Q_yg^{-1}=E_{24}+E_{35}+\frac x2E_{36}-E_{42}-E_{53}-\frac y2E_{63},
\]
so for every tangent vector $\dot g$ to the family,
\[
 \tr(Q_yg^{-1}\dot g)=\dot g_{42}+\dot g_{53}+\frac x2\dot g_{63}
 -\dot g_{24}-\dot g_{35}-\frac y2\dot g_{36}=0.
\]
Indeed, the entries $g_{42}=g_{24}=g_{53}=g_{35}=1$ and
$g_{63}=g_{36}=0$ do not depend on $(\zeta,x,y)$.
Lemma~\ref{lem:metric-criterion} applies.
\end{proof}

\begin{lemma}\label{lem:secondU}
The layer $C_0\setminus Z$ is unipotently $\chi$-incompatible.
\end{lemma}
\begin{proof}
Parametrize $C_0\setminus Z$ by $g(A,B)$ with $c=n=0$, where, with
$t=dq$, the numbers $t/a$ and $t/m$ are not both real. Recall that
$a,d,m,q\ne0$ on $C_0$. Put
\[
 \alpha_0=\Imm(t/a),\qquad\mu_0=\Imm(t/m),\qquad
 \Delta_0=\alpha_0^2+\mu_0^2>0.
\]
For $\nu_1,\nu_2\in\R$, let $s$ and $v$ be the elements defined
before~\eqref{eq:secondchar}, with $u_0=i\nu_1$, $r=i\nu_2$ and $z=0$.
They lie in $S$ and satisfy $sg(A,B)=g(A,B)v^t$. They are unipotent, because the
Levi blocks $H_0$ and $k$ used there are unipotent upper triangular.
By~\eqref{eq:secondchar},
\[
 \chi(s,v^{-1})=\psi\Bigl(\Tr\Bigl(t\Bigl(\frac{i\nu_2}m-\frac{i\nu_1}a
 \Bigr)\Bigr)\Bigr)
 =\exp\bigl(2i\kappa(\nu_1\alpha_0-\nu_2\mu_0)\bigr).
\]
The choice $\nu_1=\pi\alpha_0/(2\kappa\Delta_0)$ and
$\nu_2=-\pi\mu_0/(2\kappa\Delta_0)$ depends smoothly on $(A,B)$ and makes
this character equal to $e^{i\pi}=-1$.
Lemma~\ref{lem:unipotent-criterion} applies.
\end{proof}

\begin{lemma}\label{lem:Zmetric}
The layer $Z$ is metrically proper.
\end{lemma}
\begin{proof}
Parametrize $Z$ by the diagonal family~\eqref{eq:diagonalAB}: $g=g(A,B)$
with $A=\diag(a,d)$, $B=\diag(m,q)$ and $t/a,\,t/m\in\R^\times$, where
$t=dq$. Put
\begin{equation}\label{eq:boundary-QZ}
 Q_Z=i\diag(0,-1,1,1,-1,0).
\end{equation}
By~\eqref{eq:lie-algebra}, $Q_Z\in\g$. Its three $2\times2$ diagonal
blocks sum to zero, and $B(Q_Z,Q_Z)=-4$. Since $g$ is the permutation
matrix $w$ multiplied on both sides by diagonal matrices,
\[
 gQ_Zg^{-1}=wQ_Zw=i\diag(1,-1,0,0,-1,1),
\]
whose three $2\times2$ diagonal blocks also sum to zero. Moreover
\[
 Q_Zg^{-1}=i\bigl(a^{-1}E_{41}-t^{-1}E_{22}+\bar aE_{36}-\bar tE_{55}\bigr),
\]
and $g_{14}=a$, $g_{22}=t$, $g_{63}=\bar a^{-1}$, $g_{55}=\bar t^{-1}$.
Hence, for every tangent vector $\dot g$ to the family,
\[
 \tr(Q_Zg^{-1}\dot g)
 =i\left(\frac{\dot a}a-\frac{\dot t}t-\frac{\dot{\bar a}}{\bar a}
   +\frac{\dot{\bar t}}{\bar t}\right)
 =2\Imm\left(\frac{\dot t}t-\frac{\dot a}a\right)=0,
\]
since $\dot t/t-\dot a/a$ is the logarithmic derivative of the real
function $t/a$. Lemma~\ref{lem:metric-criterion} applies.
\end{proof}

\begin{proposition}\label{prop:boundary}
Let $\gamma\in\C$. Every $f\in\cgen_\chi(G';T)$ with $\Omega f=\gamma f$
vanishes.
\end{proposition}
\begin{proof}
We apply Lemma~\ref{lem:real-filtration}
\cite[Lemma~3.7]{JSZ} with
\[
 X=G',\qquad L=\HH,\qquad\theta=\chi,\qquad D=\Omega-\gamma.
\]
The operator $D$ is of Laplacian type for the metric~\eqref{eq:metric},
as explained after Lemma~\ref{lem:metric-support}. The filtration
$T\supset C\supset C_0\supset Z\supset\varnothing$ consists of closed
$\HH$-stable subsets of $G'$, and its successive differences are smooth.
They are unipotently $\chi$-incompatible
(Lemma~\ref{lem:firstU}), metrically proper (Lemma~\ref{lem:Cmetric}),
unipotently $\chi$-incompatible (Lemma~\ref{lem:secondU}) and metrically
proper (Lemma~\ref{lem:Zmetric}), respectively. All hypotheses of the
filtration lemma therefore hold. It removes the four differences in
this order, so the support of $f$ is contained successively in $C$,
$C_0$, $Z$ and the empty set. Hence $f=0$.
\end{proof}

\subsection{Proof of Theorem~\ref{thm:distribution} over
\texorpdfstring{$\R$}{R}}\label{subsec:real-proof}

Let $f$ be an $(\HH,\chi)$-equivariant tempered generalized function on
$G$ with $\Omega f=\gamma f$, and put $f^-=f-\tau^*f$. Transpose is a
Nash diffeomorphism of $G$, so $\tau^*f$ is tempered. Transpose
interchanges the two factors of $\HH$, which carry the same character
$\eta$, so $\tau^*f$ is again $(\HH,\chi)$-equivariant. It also commutes
with $\Omega$ (Section~\ref{subsec:real-tools}). Thus $f^-$ is a tempered
$(\HH,\chi)$-equivariant generalized function with $\tau^*f^-=-f^-$ and
$\Omega f^-=\gamma f^-$.

By Proposition~\ref{prop:open}, $f^-|_O=0$. If $H=J$, then $f^-|_{G'}$ is
supported in $T=G'\setminus O$, so Proposition~\ref{prop:boundary} gives
$f^-|_{G'}=0$, and Proposition~\ref{prop:casimir-extension}(1) gives
$f^-=0$. If $H=I_2$, Proposition~\ref{prop:casimir-extension}(2) gives
$f^-=0$ directly. In both cases $f=\tau^*f$, which proves the real case
of Theorem~\ref{thm:distribution}. Temperedness was used only through
Proposition~\ref{prop:open}.

\subsection{The multiplicity bound}\label{subsec:real-multiplicity}

Let $\pi$ be an irreducible Casselman--Wallach representation of $G$; the
$\Hom$ spaces below consist of continuous linear functionals. We apply
Proposition~\ref{prop:real-GK} \cite[Proposition~10.1]{JSZ} with
$G_0=G$, $\sigma=\tau$, and the subgroups and characters
in~\eqref{eq:GK-characters}. Its equivariance condition is
exactly~\eqref{eq:equiv}. Since $\Omega\in\mathcal Z_G$, every joint
eigenvector occurring in that criterion satisfies $\Omega f=\gamma f$
for some $\gamma\in\C$. Theorem~\ref{thm:distribution} therefore gives
the required $\tau$-invariance, and the criterion yields
\[
 \dim\Hom_S(\pi,\eta)\,
 \dim\Hom_{\tau(S)}(\pi^\vee,\eta\circ\alpha)\leq1,
 \qquad\alpha(g)=\tau(g)^{-1}.
\]
By Lemma~\ref{lem:dual}, $\pi^\vee\simeq\pi\circ\alpha$. Since $\alpha$
restricts to an isomorphism $\tau(S)\to S$, a continuous linear
functional on the space of $\pi$ is $(\tau(S),\eta\circ\alpha)$-equivariant
for $\pi\circ\alpha$ exactly when it is $(S,\eta)$-equivariant for $\pi$.
Hence
\[
 \Hom_{\tau(S)}(\pi^\vee,\eta\circ\alpha)
 \simeq\Hom_{\tau(S)}(\pi\circ\alpha,\eta\circ\alpha)
 =\Hom_S(\pi,\eta).
\]
Therefore $(\dim\Hom_S(\pi,\eta))^2\leq1$, which proves
Theorem~\ref{thm:main} over $\R$.

\end{document}